\documentclass[a4paper]{amsart}

\usepackage[utf8]{inputenc}
\usepackage[T1]{fontenc}
\usepackage{lmodern}
\usepackage{microtype}
\usepackage{amscd}
\usepackage[pdftitle={...},
 pdfauthor={Alcides Buss, Siegfried Echterhoff, Rufus Willett},
 pdfsubject={Mathematics}
]{hyperref}
\usepackage[lite]{amsrefs}

\usepackage{verbatim}

\usepackage{amsfonts}

\usepackage{amsthm}

\usepackage{amssymb}

\usepackage{amsmath}

\usepackage{mathabx}

\usepackage{enumerate}

\usepackage[all]{xy}

\usepackage{graphicx}

\usepackage{hyperref}

\newcommand{\N}{\mathbb{N}}

\newcommand{\C}{\mathbb{C}}

\newcommand{\Manoa}{M\=anoa}

\newcommand{\Hawaii}{Hawai\kern.05em`\kern.05em\relax i}

\newcommand{\inj}{\mathrm{inj}}

\newcommand{\red}{r}

\usepackage{xcolor}

\DeclareMathOperator{\Aut}{Aut}

\DeclareMathOperator{\Prim}{\mathrm{Prim}}

\newcommand{\op}{{\operatorname{op}}}

\newcommand{\free}{\mathbb{F}}
\newcommand{\Bd}{\mathbb{B}}

\newcommand*{\nb}{\nobreakdash}
\newcommand*{\Star}{\(^*\)\nobreakdash-}

\newcommand*{\K}{\mathbb K}
\newcommand*{\M}{\mathcal{M}} 

\newcommand*{\cont}{C}
\newcommand*{\contz}{\cont_0}
\newcommand*{\contc}{\cont_c}

\newcommand*{\id}{\textup{id}}
\newcommand*{\Ad}{\textup{Ad}}

\newcommand*{\U}{\mathcal U}

\newcommand*{\sbe}{\subseteq} 

\newcommand*{\cstar}{\texorpdfstring{$C^*$\nobreakdash-\hspace{0pt}}{*-}}

\newcommand*{\into}{\hookrightarrow}
\newcommand*{\onto}{\twoheadrightarrow}
\renewcommand*{\max}{\mathrm{max}}

\theoremstyle{plain}
\newtheorem{theorem}{Theorem}[section]
\newtheorem{lemma}[theorem]{Lemma}
\newtheorem{corollary}[theorem]{Corollary}
\newtheorem{proposition}[theorem]{Proposition}

\newtheorem{definition-theorem}[theorem]{Definition / Theorem}

\newtheorem*{conjecture*}{Conjecture}
\newtheorem*{theorem*}{Theorem}

\theoremstyle{definition}
\newtheorem{definition}[theorem]{Definition}
\newtheorem{example}[theorem]{Example}

\theoremstyle{remark}
\newtheorem{remark}[theorem]{Remark}

\newtheorem*{example*}{Example}  
\newtheorem*{remark*}{Remark}

\begin{document}
\title{Injectivity, crossed products, and amenable group actions}
\author{Alcides Buss}
\email{alcides.buss@ufsc.br}
\address{Departamento de Matem\'atica\\
 Universidade Federal de Santa Catarina\\
 88.040-900 Florian\'opolis-SC\\
 Brazil}

\author{Siegfried Echterhoff}
\email{echters@uni-muenster.de}
\address{Mathematisches Institut\\
 Westf\"alische Wilhelms-Universit\"at M\"un\-ster\\
 Einsteinstr.\ 62\\
 48149 M\"unster\\
 Germany}

\author{Rufus Willett}
\email{rufus@math.hawaii.edu}
\address{Mathematics Department\\
 University of \Hawaii~at \Manoa\\
Keller 401A \\
2565 McCarthy Mall \\
 Honolulu\\
 HI 96822\\
USA}

\begin{abstract}
This paper is motivated primarily by the question of when the maximal and reduced crossed products of a $G$-$C^*$-algebra agree (particularly inspired by results of Matsumura and Suzuki), and the relationships with various notions of amenability and injectivity.  We give new connections between these notions.  Key tools in this include the natural equivariant analogues of injectivity, and of Lance's weak expectation property: we also give complete characterizations of these equivariant properties, and some connections with injective envelopes in the sense of Hamana.
\end{abstract}

\maketitle

\tableofcontents

\section{Introduction}

This paper studies the relationship between various notions of amenability of actions of a group $G$ on a $C^*$-algebras $A$ and the `weak containment' question of whether $A\rtimes_{\max}G=A\rtimes _r G$.  We were motivated by trying to elucidate relationships between the following works:
\begin{itemize}
\item recent interesting examples of Suzuki \cite{Suzuki:2018qo} showing a disconnect between notions of amenable actions and the equality $A\rtimes_{\max}G=A\rtimes _r G$;
\item Matsumura's work \cite{Matsumura:2012aa} relating the property $A\rtimes_{\max}G=A\rtimes _r G$ to amenability (at least in the presence of exactness of the acting group);
\item the extensive work of Anantharaman-Delaroche on amenable actions, most relevantly for this paper in \cite{Anantharaman-Delaroche:1987os} and \cite{Anantharaman-Delaroche:2002ij};
\item the seminal theorem of Guentner-Kaminker \cite{Guentner:2022hc} and Ozawa \cite{Ozawa:2000th} relating exactness to existence of amenable actions;
\item our study of the so-called maximal injective crossed product \cite{Buss:2018nm} and the connections to equivariant versions of injectivity;
\item Hamana's classical study of equivariant injective envelopes \cite{Hamana:1985aa} and the recent important exploitation of these ideas by Kalantar and Kennedy in their work on the Furstenberg boundary \cite{Kalantar:2014sp}.
\end{itemize}
Our goal was to try to bridge connections between some of this in a way that we hope systematises some of the existing literature a little better, as well as solving some open problems.  As the material is by nature somewhat technical, we will refrain from giving precise statements in this introduction, but just some flavour.  Our results are perhaps most easily explained by discussing the contents of the paper, which we now do.

In Section \ref{basic sec}, we recall two notions of amenable action on a $C^*$-algebra due to Claire Anantharaman-Delaroche, and recall the theorems of Guentner-Kaminker and Ozawa on exactness.  We also make a simple observation based on this (Proposition \ref{G-exact-char-amenable}) that will be used over and over again in one form or another throughout the paper: roughly, this says that given an exact group $G$, for $A$ to be amenable, it is sufficient that there exists an equivariant ucp map from $\ell^\infty(G)$ to the centre of the multiplier algebra $\M(A)$ (or the center of the double dual $A^{**}$).  In Section \ref{suzuki sec}, we briefly recall Suzuki's examples.  We then show that while they do not have Anantharaman-Delaroche's strong amenability property, they do satisfy a version of Exel's  approximation property and are therefore amenable in the sense of Anantharaman-Delaroche.  The key idea here is to drop a precise centrality condition in favour of some form of `quasi-centrality'.

Motivated by Suzuki's examples, in Section \ref{weak con sec} we try to find reasonable conditions on $A$ that characterize when $A\rtimes_{\max} G=A\rtimes_\red G$ in general.  We are able to do this (at least in the presence of exactness) in terms of injectivity-type conditions on representations (Corollary \ref{inj lem}), and in terms of a weak amenability-like condition that we call commutant amenability (Theorem \ref{com amen the}).  We should explicitly say that while these results seem theoretically useful, they have the drawback that they are quite difficult to check in concrete examples without knowing something a priori stronger.  In Section \ref{comm case sec}, we continue our study of weak containment.  Following ideas of Matsumura, we now bring the Haagerup standard form of the double dual into play, and use this to get more precise results on weak containment somewhat generalizing Matusmura's: the most satisfactory of these are in the setting of actions of exact groups on commutative $C^*$-algebras (Theorem \ref{com ex the}), but we also have partial results for noncommutative algebras, and non-exact groups.  

In Section \ref{exact sec}, we go back to the relationship with exactness.  Thanks to the above-mentioned work of Guentner-Kaminker and Ozawa, it is well-known that a group $G$ is exact if and only if it admits an amenable action on a compact space.  It is thus natural to ask whether the analogous result holds for amenable actions on unital, possibly noncommutative $C^*$-algebras, i.e.\ is it true that $G$ is exact if and only if it admits an amenable action on a unital $C^*$-algebra? The answer is (clearly) `yes' if `amenable' in this statement means what Anantharaman-Delaroche calls strong amenability, but this is less clear in general.  We show in fact (see Theorem \ref{com amen ex}) that the answer is `yes' in the strong sense that a group $G$ is exact if and only if it admits a commutant amenable action on a unital $C^*$-algebra; commutant amenability is the weakest reasonable notion of amenability that we know of.

Section \ref{inj and wep sec} studies equivariant versions of injectivity, and of Lance's weak expectation property; these are used throughout the paper, but here we look at them more seriously.  In particular, we give complete characterizations of when a (unital) $G$-algebra has these properties in terms of amenability and of the underlying non-equivariant versions (Theorems \ref{g wep vs wep} and \ref{g inj vs inj}); again, exactness turns out to play a fundamental (and quite subtle) role.  Finally, in Section \ref{hamana sec} we discuss the relationship of our notion of injectivity to that introduced by Hamana (fortunately, they turn out to be the same), and to his equivariant injective envelopes.  We give applications of this material to a conjecture of Ozawa on nuclear subalgebras of injective envelopes (Corollary \ref{oz con cor}), and to the existence of amenable actions on injective envelopes (Theorem \ref{inj env ex}).  In both cases, our results generalize work of Kalantar and Kennedy.

The initial motivation for this paper grew out of the relation of injectivity and the weak expectation property with the maximal injective crossed product functor as studied in \cite{Buss:2018nm}. It is interesting to note that this functor is, in a sense, `dual' to the minimal exact crossed product which was studied by the authors in \cite{Buss:2018pw} and has a close relation to the Baum-Connes conjecture. The present paper shows that the 
maximal injective crossed product nicely relates to amenability.

This paper started during visits of the first and third authors to the University of Münster. The first and third authors are grateful for the warm hospitality provided by the second author and the operator algebra group of that university.

\section{Notation, basic definitions and preliminaries}\label{basic sec} 

We use the following notation.  The abbreviations ucp and ccp stand for `unital completely positive' and `contractive completely positive' respectively.  Throughout the paper, $G$ always refers to a \emph{discrete} group.  A \emph{$G$-algebra} will always refer to a $C^*$-algebra $A$ equipped with an action of $G$ by $*$-automorphisms (we will not really discuss any algebras that are not $C^*$-algebras, so this should not lead to confusion).  A \emph{$G$-space} will always refer to a locally compact Hausdorff space $X$ equipped with an action of $G$ by homeomorphisms; note that if $X$ is a $G$-space, then $A=C_0(X)$ is a $G$-algebra and vice versa.  Generally, we will not explicitly introduce notation for the action unless it is needed.

Given a $G$-algebra $A$, we will equip various associated algebras with the canonically induced $G$-actions without explicitly stating this.  Thus for example this applies to the multiplier algebra $\M(A)$, the double dual $A^{**}$, the opposite algebra $A^{\op}$, and the centre $Z(A)$.  Given $G$-algebras $A$ and $B$, we will always equip the spatial and maximal tensor products $A\otimes B$ and $A\otimes_{\max}B$ with the associated diagonal $G$-action unless explicitly stated otherwise.  

A map $\phi:A\to B$ between sets with $G$-actions $\alpha$ and $\beta$ is \emph{equivariant} if $\phi(\alpha_g(a))=\beta_g(\alpha(a))$ for all $a\in A$ and $g\in G$.  We will also call equivariant maps \emph{$G$-maps} and allow other similar modifiers as appropriate (for example `ucp $G$-map', `$G$-embedding', ...).   Relatedly, a $G$-subalgebra $A$ of $B$ will be a C*-subalgebra that is invariant under the $G$-action (and is therefore a $G$-algebra in its own right).  Equivalently, we can think of a $G$-subalgebra $A$ of a $G$-algebra $B$ as a $G$-algebra $A$ equipped with a $G$-embedding $\iota:A\into B$.

The $C^*$-algebra $\ell^\infty(G)$ will play a special role in this paper.  It is always considered as a $G$-algebra via the (left) translation action defined by 
$$
(\gamma_gf)(h):=f(g^{-1}h).
$$

The following definitions are (to the best of our knowledge) due to Claire Anantharaman-Delaroche: the definition of positive type functions is from \cite[Definition~2.1]{Anantharaman-Delaroche:1987os}, amenability of an action is from \cite[Definition~4.1]{Anantharaman-Delaroche:1987os}. 
If $A$ is commutative or unital, then our notion of strong amenability as defined below is equivalent to a notion of strong amenability due to  Ananthraman-Delaroche \cite[Definition~6.1]{Anantharaman-Delaroche:2002ij} (see Lemma~\ref{lem:strong-amen=Claire} below). In general our notion of strong amenability could possibly be weaker than the one introduced by Anantharaman-Delaroche.

\begin{definition}\label{def-amenable}
Let $A$ be a $G$-algebra with associated action $\alpha$.  

A function $\theta:G\to A$ is \emph{positive type} if for any finite subset $\{g_1,...,g_n\}$ of $G$, the matrix 
$$
\big(\alpha_{g_i}(g_i^{-1}g_j)\big)_{i,j}\in M_n(A)
$$
is positive. 

We say that $A$ is \emph{amenable} if there exists a net $(\theta_i:G\to Z(A^{**}))_{i\in I}$ of positive type functions such that:
\begin{enumerate}[(i)]
\item each $\theta_i$ is finitely supported;
\item for each $i$, $\theta_i(e)\leq 1$;
\item for each $g\in G$, $\theta_i(g)\to 1$ ultraweakly as $i\to\infty$.
\end{enumerate}

And we call $A$ \emph{strongly amenable} if there exists a net $(\theta_i:G\to Z\M(A))_{i\in I}$ of positive type functions such that:
\begin{enumerate}[(i)]
\item each $\theta_i$ is finitely supported;
\item for each $i$, $\theta_i(e)\leq 1$;
\item for each $g\in G$, $\theta_i(g)\to 1$ strictly as $i\to\infty$.
\end{enumerate}
\end{definition}

\begin{remark} Our notion of an amenable $G$-algebra should not be mistaken by the notion of an amenable (i.e., nuclear) $C^*$-algebra.
The terminology for amenable actions is unfortunately not completely consistent in the literature.  The notion of strong amenability also appears as \cite[Definition 4.3.1]{Brown:2008qy} in the text of Brown and Ozawa (although only in the special case of unital $G$-algebras).  There it is just called \emph{amenability}.  The notion of amenability as defined above is equivalent to what is called \emph{weak amenability} in \cite[Definition~6.1]{Anantharaman-Delaroche:2002ij} (see \cite[Proposition~6.4]{Anantharaman-Delaroche:2002ij}) for nuclear $G$-algebras.  However, it is not clear if this is true in the non-nuclear case, so we will keep to the terminology `amenability'. Observe that strong amenability always implies amenability,
since the canonical inclusion $\M(A)\into A^{**}$ is strict to ultraweak continuous.

For commutative $G$-algebras $A=\contz(X)$, the notions of amenability and strong amenability are the same, and both are equivalent to amenability of the $G$-space $X$, see \cite[Th\'{e}or\`{e}me 4.9 and Remarque~4.10]{Anantharaman-Delaroche:1987os}.  
However, we will see in Section \ref{suzuki sec} below that they are \emph{not} the same in general even if we restrict to nuclear, unital $C^*$-algebras.

The following lemma will get used many times in the paper.

\begin{lemma}\label{ucp pass}
Let $A$ and $B$ be $G$-algebras and suppose there exists a strictly continuous ucp $G$-map $\phi:Z\M(A)\to Z\M(B)$. Then,  if $A$ is strongly amenable, so is $B$.
\end{lemma}

\begin{proof}
If $(\theta_i)$ is a net with the properties required to show strong amenability of $A$, then it is not difficult to check that $(\phi\circ \theta_i)$ has the properties required to show strong amenability of $B$.
\end{proof}

\begin{remark}
Notice that the above result applies, in particular, if $A$ is unital and there is a ucp $G$-map $Z(A)\to Z\M(B)$. 
\end{remark}

In \cite[Definition~6.1]{Anantharaman-Delaroche:2002ij}, Anantharaman-Delaroche defines a (possibly noncommutative) $G$-algebra $A$ to be {\em strongly amenable} if there exists an amenable $G$-space $X$ and a nondegenerate $G$-equivariant $*$-homomorphism $\Phi:C_0(X)\to Z\M(A)$. 

\end{remark}
\begin{lemma}\label{lem:strong-amen=Claire}
Every strongly amenable $G$-algebra in the sense of \cite[Definition~6.1]{Anantharaman-Delaroche:2002ij} is strongly amenable in the sense of  Definition \ref{def-amenable} above. If the $G$-algebra $A$ is unital or commutative, then both notions of strong amenability coincide.
\end{lemma}
\begin{proof}
The first assertion follows from Lemma~\ref{ucp pass} as every nondegenerate \Star{}homo\-mor\-phism $\contz(X)\to Z\M(A)$ extends to a strictly continuous \Star{}homomorphism $\M(\contz(X))\to Z\M(A)$. If $A$ is commutative, both definitions coincide by \cite[Proposition 6.3]{Anantharaman-Delaroche:2002ij}, so let us assume now that $A$ is unital
and let $X$ be the Gelfand dual of $Z(A)$. Then strong amenability in the sense of Definition \ref{def-amenable}  implies that $C(X)$ is an amenable 
$G$-algebra, thus $X$ is an amenable $G$-space by \cite[Th\'{e}or\`{e}me 4.9 and Remarque~4.10]{Anantharaman-Delaroche:1987os}. 
Since the inclusion $C(X)\cong Z(A)\into A$ is unital (hence nondegenerate), $A$ is strongly amenable in the sense of 
\cite[Definition~6.1]{Anantharaman-Delaroche:2002ij}.
\end{proof}

The following important theorem is due to Ozawa \cite{Ozawa:2000th} and (partially) Guentner-Kaminker \cite{Guentner:2022hc}.  For background on exact groups, see for example \cite{Willett:2009rt} or \cite[Chapter 5]{Brown:2008qy}.

\begin{theorem}\label{ozawa the}
For a discrete group $G$, the following are equivalent:
\begin{enumerate}[(i)]
\item $G$ is exact;
\item the canonical $G$-action on $\ell^\infty(G)$ is strongly amenable.  \qed
\end{enumerate}
\end{theorem}

\begin{proposition}\label{G-exact-char-amenable}
Let $G$ be an exact group and $A$ be a $G$-algebra.  The following hold.
\begin{enumerate}[(i)]
\item If there exists a ucp $G$-map $\ell^\infty(G)\to Z(\M(A))$, then $A$ is strongly amenable.
\item There exists a ucp $G$-map $\ell^\infty(G)\to Z(A^{**})$ if and only if $A$ is amenable.
\end{enumerate}
As a consequence, $A$ is amenable if and only if $A^{**}$ is strongly amenable.
\end{proposition}

\begin{proof}
The first part follows immediately from Lemma \ref{ucp pass} and Theorem \ref{ozawa the}.  

For the second part, let $\alpha$ be the action on $A^{**}$, and equip $\ell^\infty(G,Z(A^{**}))$ with the $G$-action defined by 
$$
(\widetilde{\alpha}_gf)(h):=\alpha_g(f(g^{-1}h)).
$$
If $A$ is amenable, we have by \cite[Th\'{e}or\`{e}me 3.3]{Anantharaman-Delaroche:1987os} that there is a ucp $G$-map $P\colon \ell^\infty(G,Z(A^{**}))\to Z(A^{**})$.  Composing this with the canonical unital $G$-embedding $\ell^\infty(G)\into \ell^\infty(G,Z(A^{**}))$ gives the desired ucp $G$-map $\ell^\infty(G)\to Z(A^{**})$.

Conversely, assume that there is a ucp $G$-map $\ell^\infty(G)\to Z(A^{**})$. Since $G$ is exact, its translation action on $\ell^\infty(G)$ is strongly amenable, whence $A^{**}$ is strongly amenable by Lemma \ref{ucp pass}, and so $A$ is amenable, since convergence in norm implies ultraweak convergence.
\end{proof}

\begin{remark}\label{rem-norm} Note that it follows from the above proof that for exact groups $G$ the ultraweak convergence 
 in item (iii) of the definition of an amenable $G$-algebra (see Definition  \ref{def-amenable}) can be replaced by norm convergence.
\end{remark}

We will need some equivalent versions of amenability and strong amenability 

\begin{definition}\label{standard mod}
Let $A$ be a $G$-algebra with action $\alpha$.  Define $\ell^2(G,A)$ to be the collection of all functions $\xi:G\to A$ such that 
$$
\sum_{g\in G}\xi(g)^*\xi(g)
$$
converges in the norm of $A$.  Equip $\ell^2(G,A)$ with the $A$-valued inner product defined by 
$$
\langle \xi,\eta\rangle:=\sum_{g\in G}\xi(g)^*\eta(g),
$$
the norm defined by 
$$
\|\xi\|_2:=\sqrt{\|\langle \xi,\xi\rangle\|_A}
$$
and the $G$-action $\widetilde{\alpha}$ defined by 
$$
(\widetilde{\alpha}_h\xi)(g):=\alpha_{h}(\xi(h^{-1}g)).
$$
\end{definition}

\begin{lemma}\label{l2 amen}
Let $A$ be a $G$-algebra.  
\begin{enumerate}[(i)]
\item $A$ is \emph{amenable} if  and only if there exists a net $(\xi_i:G\to Z(A^{**}))_{i\in I}$ of functions such that:
\begin{enumerate}[(i)]
\item each $\xi_i$ is finitely supported;
\item for each $i$, $\langle \xi_i,\xi_i\rangle\leq 1$;
\item for each $g\in G$, $\langle \xi_i,\widetilde{\alpha}_g\xi_i\rangle\to 1$ ultraweakly as $i\to\infty$.
\end{enumerate}
\item $A$ is \emph{strongly amenable} if and only if  there exists a net $(\xi_i:G\to Z\M(A))_{i\in I}$ of functions such that:
\begin{enumerate}[(i)]
\item each $\xi_i$ is finitely supported;
\item for each $i$, $\langle \xi_i,\xi_i\rangle\leq 1$;
\item for each $g\in G$, $\langle \xi_i,\widetilde{\alpha}_g\xi_i\rangle\to 1$ strictly as $i\to\infty$.
\end{enumerate}
\end{enumerate}
\end{lemma}

\begin{proof}
The proof is essentially contained in that of \cite[Th\'{e}or\`{e}me 3.3]{Anantharaman-Delaroche:1987os} and the proofs are essentially the same in both cases, so we just sketch the idea in the amenable case.  If $(\xi_i)$ is a net as in the statement of the lemma, then 
$$
\theta_i:G\to Z(A^{**}), \quad g\mapsto \langle \xi_i,\widetilde{\alpha}_g\xi_i\rangle
$$
satisfies the properties needed to check amenability.  Conversely, if $(\theta_i)$ is as in the definition of amenability, then the fact that each $\theta_i$ is positive type and finitely supported means we can apply a GNS-type construction for each $i$ as in \cite[Proposition 2.5]{Anantharaman-Delaroche:1987os} to get a vector $\eta_i\in \ell^2(G,Z(A^{**}))$ such that $\theta_i(g)=\langle \eta_i,\widetilde{\alpha}_g\eta_i\rangle$ for all $g\in G$.  Using that the collection of finitely supported elements is norm dense in $\ell^2(G,Z(A^{**}))$ and appropriately approximating each $\eta_i$ by some $\xi_i$ gives the result.
\end{proof}

\section{Suzuki's examples}\label{suzuki sec}

In \cite{Suzuki:2018qo}, Yuhei Suzuki produces (amongst other things) a very striking class of examples.  Let $G$ be a countable, exact, non-amenable group.  Suzuki shows in \cite{Suzuki:2018qo}*{Proposition B} that there exists a simple, unital, separable, nuclear $G$-algebra $A$ such that $A\rtimes_r G=A\rtimes_{\max} G$.  As $A$ is simple and unital, its center is just scalar multiples of the unit, so $A$ cannot be strongly amenable: if it were, $G$ would necessarily be amenable.  It is well-known that strong amenability implies equality of the maximal and reduced crossed product $C^*$-algebras, but the converse had been an open question.

Now, it is also known \cite[Proposition 4.8]{Anantharaman-Delaroche:1987os} that if $A$ is an amenable $G$-algebra, then $A\rtimes_r G=A\rtimes_{\max} G$.  The converse is again open, and so it is natural to ask if Suzuki's examples are amenable.  The answer turns out to be yes: one way to see this is to note that Suzuki's examples arise as a direct limit
$$
A\rtimes_\red G=\lim_n (A_n\rtimes_\red G)
$$  
with each $A_n$ a strongly amenable nuclear $G$-subalgebra of $A$; as $A_n$ is (strongly) amenable, $A_n\rtimes G$ is nuclear by \cite[Th\'{e}or\`{e}me 4.5]{Anantharaman-Delaroche:1987os}, whence $A\rtimes G$ is nuclear, and so $A$ is amenable by \cite[Th\'{e}or\`{e}me 4.5]{Anantharaman-Delaroche:1987os} again.

While we guess Suzuki (and others) are aware of this, it does not seem to have been explicitly recorded in his paper.  Suzuki's examples seem to be the first known examples of $G$-algebras with an amenable action that is not strongly amenable.

In the rest of this section, we will give another approach to amenability of Suzuki's algebras, partly as it is more concrete, and partly as we suspect it will be useful in other contexts.   This involves another variant of amenability in the form of an approximation property.  Although this variant will not be used in the rest of the paper, it seems a natural notion so worth including.  It is also the strongest `amenability-type' condition that we could show that Suzuki's examples satisfy, and so seems worthwhile from that point of view.

\begin{definition}\label{def:QAP}
A $G$-algebra $A$ has the \emph{quasi-central approximation property} (QAP) if there is a net $(\xi_i\colon G\to \M(A))_{i\in I}$ of finitely supported functions satisfying 
\begin{enumerate}[(i)]
\item $\langle \xi_i,\xi_i\rangle\leq 1$ for all $i$;
\item $\langle \xi_i,\tilde\alpha_g\xi_i\rangle$ converges strictly to $1$ in $\M(A)$ for all $g\in G$;
\item $\|\xi_ia - a\xi_i\|_2\to 0$ for all $a\in A$.
\end{enumerate} 
\end{definition}

One can actually assume that the functions $\xi_i$ in the definition of the QAP take their values in $A$:
\begin{lemma}\label{lem-QAP} 
Suppose that the $G$-algebra $A$ has the QAP. Then the functions $(\xi_i)_{i\in I}$ in the definition of the QAP 
can be chosen to take their values in $A$, i.e., 
there exists a net $(\xi_i\colon G\to A)_{i\in I}$ which satisfies 
conditions (i), (ii), and (iii) of Definition \ref{def:QAP}.
\end{lemma}
\begin{proof} Let $(\eta_i: G\to \M(A))_{i\in I}$ be a net as in Definition \ref{def:QAP} and let $(e_j)_{i\in J}$ be a quasi-central 
approximate unit of $A$, i.e., we have $\|e_ja-ae_j\|\to 0$ for all $a\in A$, and $0\leq e_j\leq 1$ for all $j\in J$. 
 Define $\xi_{i,j}:G\to A$ by $\xi_{i,j}(g):=\eta_i(g) e_j$. Then
$\langle \xi_{i,j},\xi_{i,j}\rangle=e_j\langle \eta_i, \eta_i\rangle e_j\leq 1$ for all $(i,j)\in I\times J$ and 
$$\langle \xi_{i,j}, \tilde\alpha_g\xi_{i,j}\rangle=e_j\langle \eta_i,\tilde\alpha_g\eta_i\rangle\alpha_g(e_j)\to 1$$
in the strict topology of $A$, using the fact that multiplication is strictly continuous on bounded subsets of $\M(A)$.
Thus  $(\xi_{i.j})_{(i.j)\in I\times J}$  satisfies conditions (i) and (ii) of Definition \ref{def:QAP}. To check condition (iii) 
let $a\in A$ be fixed. Then
\begin{align*}
\| \xi_{i,j} a-a\xi_{i,j}\|_2&=\|\eta_i e_j a- a\eta_i e_j\|_2\\
&\leq \|\eta_i e_j a-\eta_i a e_j\|_2+\|\eta_ia e_j- a\eta_i e_j\|_2\\
&\leq \|\eta_i\|_2 \cdot \|e_j a- a e_j\|+\|\eta_ia-a\eta_i\|_2\cdot \|e_j\|\\
&\leq  \|e_j a- a e_j\|+\|\eta_ia-a\eta_i\|_2\to 0.
\end{align*}
This finishes the proof.
\end{proof}

It follows from Lemma \ref{l2 amen} that if $A$ is strongly amenable, then it has the QAP.  On the other hand, it follows from Lemma \ref{lem-QAP} that the QAP implies the so-called \emph{approximation property} of Exel (see \cite[Definition 20.4]{Exel:2014rp}), and therefore that the QAP implies amenability by the results of \cite[Theorem 6.11 and Corollary 6.16]{Abadie:2019kc}.  To summarize, we have the following implications in general:
$$
\text{strong amenability} \;\; \Rightarrow\;\; \text{QAP} \;\; \Rightarrow\quad \text{amenability}.
$$
We will soon show that Suzuki's examples have the QAP 
(but are not strongly amenable), whence the first implication above is not reversible.  We do not know if the second is reversible.  The key point for showing Suzuki's examples have the QAP is as follows.

\begin{proposition}\label{pro:QAP-limit-1}
Assume  the $G$-algebra $A$ 
is the inductive limit of a sequence (or net) of $G$-algebras $(A_n)_{n\in N}$.
If all $A_n$ have the QAP, then so does $A$. In particular, if all $A_n$ are strongly amenable, then $A$ is amenable.
\end{proposition}
\begin{proof} Since the QAP passes to quotients by $G$-invariant ideals, we may assume 
without loss of generality that $(A_n)_{n\in N}$ is an increasing net of $G$-algebras such that $\cup_{n\in N}A_n$ is dense in $A$.
For each $n$ let  $(\xi_{i,n})$ be a net of functions $\xi_{i,n}:G\to A_n$ satisfying the conditions of 
Definition \ref{def:QAP} for the QAP of $A_n$.  Let $\eta_{i,n}:G\to A$ denote the composition of $\xi_{i,n}$ with the 
inclusion $A_n\into A$. 
It is clear that the net $(\eta_{i,n})$ satisfies condition (i) in the definition of the QAP.
Moreover,  conditions (ii) and (iii) for $(\xi_{i,n})$ imply that 
$$a\langle \eta_{i,n},\tilde\alpha_g\eta_{i,n}\rangle\to a,\; \langle \eta_{i,n},\tilde\alpha_g\eta_{i,n}\rangle a\to a \quad \text{and}\quad
\|\eta_{i,n}a - a\eta_{i,n}\|_2\to 0$$
for all $a\in A_n$, and hence for all $a\in \cup_{n\in N} A_n$. 
Since $(\eta_{i,n})$ is uniformly bounded (by $1$) with respect to the $\ell^2$-norm, it follows that (ii) and (iii) hold for all $a\in A$.
\end{proof}

Now, let us briefly describe Suzuki's examples as given in \cite{Suzuki:2018qo}*{Proposition B}  in order to see how they fit into the above discussion.  Let $G$ be any countable exact group.  Then it is always possible to choose a second countable, compact, amenable $G$-space $X$ (i.e.\ the $G$-algebra $C(X)$ is strongly amenable) such that the $G$-action is also free and minimal (see for example \cite[Section 6]{Rordam:2010kx}).  The crossed product $A_0:= \cont(X)\rtimes G$ is therefore a simple, separable, unital and nuclear \cstar{}algebra.  Consider $A_0$ as a $G$-algebra endowed with the conjugation action, that is, the (inner) action implemented by the canonical unitaries $u_g\in C^*_r(G)\subseteq  A_0$. Observe that an inner action can never be amenable unless $G$ is amenable (or the algebra is zero). However, the diagonal $G$-action on the infinite tensor product $A:= A_0^{\otimes\N}$ has the QAP (and is therefore amenable).

Indeed, as observed by Suzuki, it is enough to realise $A$ as the direct limit of the $G$-subalgebras $A_n:= A_0^{\otimes n}\otimes\cont(X)$, which are $X\rtimes G$-algebras in the obvious way. It follows that all $A_n$ are strongly amenable, hence $A$ has the QAP
by Proposition~\ref{pro:QAP-limit-1}.

Note that Suzuki's examples also show that (unlike the QAP), strong amenability does not behave well with respect to limits.
It is not clear to us whether amenability passes to inductive limits of $G$-algebras,
but Proposition \ref{pro:QAP-limit-1}  at least shows that limits 
of strongly amenable $G$-algebras are amenable.

We close this section with a brief discussion on the amenability of some other examples of $G$-algebras $A$ 
with $A\rtimes G=A\rtimes_rG$ 
constructed by Suzuki in \cite{Suzuki:2018qo}. 
In  \cite{Suzuki:2018qo}*{Theorem A} he produces examples of actions $\alpha:G\to \Aut(A)$ of non-amenable second countable 
locally compact groups on simple 
$C^*$-algebras $A$ such that the full and reduced crossed products coincide. Since $A$ is simple, we have $Z\M(A)\cong C_b(\Prim(A))=\C$,
hence, since $G$ is not amenable,  $A$ cannot be a strongly amenable $G$-algebra. 

To see that $A$ is  amenable if $G$ is discrete,  recall that 
Suzuki constructs $A$ as a crossed product $A=C_0(X\times G)\rtimes \Gamma$
with respect to a certain minimal diagonal action, say $\beta$, of a free group $\Gamma$ such that  $X$ is a compact  amenable $\Gamma$-space
and $\beta$ commutes with the $G$-action $\tilde\rho:=\id_X\otimes \rho$ on $C_0(X\times G)$, where $\rho$ denotes the right translation action of $G$ on itself. 
Then the action $\beta$ on $C_0(X\times G)$ is amenable and therefore
 $A=C_0(X\times G)\rtimes \Gamma$ is nuclear by 
  \cite[Th\'{e}or\`{e}me 4.5]{Anantharaman-Delaroche:1987os}. Now, since the $G$-action $\tilde\rho$ commutes with $\beta$, it
  induces an action, say $\gamma$,  of $G$ on $A=C_0(X\times G)\rtimes \Gamma$, which is the action considered by Suzuki.
  To see that this action is amenable, by   \cite[Th\'{e}or\`{e}me 4.5]{Anantharaman-Delaroche:1987os} it suffices to show that $A\rtimes G$ is nuclear. But this follows from the equation 
\begin{align*}
A\rtimes_\gamma G&=\big(C_0(X\times G)\rtimes_\beta \Gamma\big)\rtimes_\gamma G=\big(C_0(X\times G)\rtimes_{\tilde\rho} G\big)\rtimes  \Gamma\\
&\cong \big(C(X)\otimes \K(\ell^2(G))\big)\rtimes\Gamma,
\end{align*}
 which is nuclear by the amenability of the action of $\Gamma$ on $C(X)\otimes \K(\ell^2(G))$, which follows from amenability of the $\Gamma$-space $X$.

A similar argument shows that Suzuki's examples of \cite{Suzuki:2018qo}*{Proposition C} are also amenable if  $G$ is  discrete (but not strongly amenable). We believe that all these examples should also have the QAP, but so far we did not succeed to give a proof.

\section{Weak containment}\label{weak con sec}

In this section, motivated in part by Suzuki's examples from \cite{Suzuki:2018qo}, and partly by issues that came up in our earlier work on exotic crossed products \cite{Buss:2018nm}, we study the question of characterizing when $A\rtimes_{\max}G$ is equal to $A\rtimes_rG$.  If $A$ satisfies this property, one sometimes says that $A$ has \emph{weak containment}, whence the title of this section.  

This question seems difficult in general: while amenability of the action is a sufficient condition by \cite[Proposition 4.8]{Anantharaman-Delaroche:1987os}, finding a `good' necessary condition that works in complete generality has proven elusive, even in the case when $A$ is commutative.

It turns out to be easier to characterize when $A\rtimes_{\max} G$ equals the so-called \emph{maximal injective} crossed product $A\rtimes_{\inj}G$, which was introduced by the current authors in \cite{Buss:2018nm}.  As proved in \cite[Proposition 4.2]{Buss:2018nm}, $A\rtimes_{\inj}G=A\rtimes_\red G$ whenever $G$ is exact, so characterizing when $A\rtimes_{\max}G=A\rtimes_{\inj}G$ is the same as characterizing when $A\rtimes_{\max}G=A\rtimes_r G$ for `most' groups that come up in `real life'.

We first recall the definition of the maximal injective crossed product from \cite[Section 3]{Buss:2018nm}.

\begin{definition}\label{inj cp def}
For a $G$-$C^*$-algebra $A$, the \emph{injective crossed product} $A\rtimes_\inj G$ is defined as the completion of $C_c(G,A)$ for the norm defined on $a\in C_c(G,A)$ by
$$
\|a\|:=\inf\{\|a\circ \phi\|_{B\rtimes_{\max}G}\mid \phi:A\to B \text{ an injective equivariant $*$-homomorphism}\}.
$$
\end{definition}

It is not immediate from the definition, but this is a \cstar{}norm.  Moreover, it defines a crossed product functor that takes injective equivariant $*$-homomorphisms to injective $*$-homomorphisms.   The following definitions are important for establishing the basic properties of $\rtimes_{\inj}$, and will be fundamental to our work in this paper.

\begin{definition}\label{inj def}
A $G$-algebra $A$ is \emph{$G$-injective} if for any $G$-embedding $A\subseteq B$, there exists an equivariant conditional expectation $P:B\to A$.

A $G$-algebra $A$ has the \emph{$G$-WEP} if for any $G$-embedding $\iota:A\into B$, there exists a ccp map $P:B\onto A^{**}$ such that $P\circ \iota:A\to A^{**}$ coincides with the canonical embedding $A\into A^{**}$.
\end{definition}

The above definition of $G$-injectivity is maybe a little non-standard.  We will see in Section \ref{hamana sec} below (see Proposition \ref{hamana vs inj}) that it is equivalent to the more usual definition due to Hamana \cite{Hamana:1985aa}.

We have the following basic lemma: this will get used several times below.

\begin{lemma}\label{uni rem}
If $B$ is an injective $G$-algebra in the sense of Definition \ref{inj def}, then it is unital.  
\end{lemma}

\begin{proof}
Let $\widetilde{B}$ be the unitisation of $B$, equipped with the unique extension of the $G$-action.  Then the natural inclusion $B\into \widetilde{B}$ admits an equivariant ccp splitting $E:\widetilde{B}\to B$, which is necessarily a conditional expectation (see for example \cite[Theorem 1.5.10]{Brown:2008qy}).  Then for any $b\in B$, 
$$
E(1)b=E(1b)=E(b)=b=E(b)=E(b1)=bE(1),
$$
so $E(1)$ is a unit for $B$.  
\end{proof}

\begin{remark}\label{inj to wep}
A $G$-injective $G$-algebra clearly has the $G$-WEP.  Moreover, if $A^{**}$ is $G$-injective, then $A$ has the $G$-WEP.  To see this, let $A\into B$ be a $G$-embedding, and extend it canonically to an embedding of double duals $A^{**}\into B^{**}$.  As $A^{**}$ is $G$-injective, this admits a splitting $P:B^{**}\to A^{**}$, and the restriction of this splitting to $B$ is the map required by the $G$-WEP.  The converse is false in general: indeed, $A=\ell^\infty(G)$ is $G$-injective by Example \ref{basic exes} just below, so always has the $G$-WEP; however, $A^{**}$ is $G$-injective if and only if $G$ is exact as will follow from Theorem \ref{ozawa the} and Theorem \ref{g inj vs inj} below.
\end{remark}

\begin{example}\label{basic exes}
Perhaps the simplest example of a $G$-injective algebra is $\ell^\infty(G)$.  Indeed, let $\ell^\infty(G)\into B$ be any $G$-embedding.  Choose a state $\phi$ on $B$ that extends the Dirac mass at the identity $\delta_e$, considered as a multiplicative linear functional $\delta_e:\ell^\infty(G)\to \C$.  Then define 
$$
P:B\to\ell^\infty(G),\quad P(b):g\mapsto \phi(\beta_{g^{-1}}(b)).
$$
It is not too difficult to see that this is a ccp $G$-map splitting the original inclusion (compare \cite[Proposition 2.2]{Buss:2018nm} for a more general result).   

Noting that $C_0(G)^{**}=\ell^\infty(G)$, Remark \ref{inj to wep} gives that $C_0(G)$ has the $G$-WEP.  It is never $G$-injective for infinite $G$, as $G$-injectivity implies unitality by Lemma \ref{uni rem}.  
\end{example}

More generally (and with essentially the same proof: see \cite[Proposition 2.2 and Remark 2.3]{Buss:2018nm}) we have the following example.

\begin{example}\label{inj ex}
Let $B$ be a $C^*$-algebra that is injective in the usual sense (i.e.\ $G$-injective where $G$ is the trivial group).  Then $\ell^\infty(G,B)$ equipped with the translation action defined by 
$$
(\gamma_g(f))(h):=f(g^{-1}h)
$$ 
is $G$-injective.  More generally, if $B$ is equipped with a $G$-action $\beta$ (but is still only assumed non-equivariantly injective) and $\ell^\infty(G,B)$ is equipped with the diagonal type action 
$$
(\widetilde{\beta}_g f)(h):=\beta_g(f(g^{-1}h)),
$$
then $\ell^\infty(G,B)$ is $G$-injective.
\end{example}

The following lemma, proved in \cite[Proposition 3.12]{Buss:2018nm}, is key to establishing these properties.

\begin{lemma}\label{g wep lem}
Let $A$ be a $G$-algebra with the $G$-WEP (in particular, $A$ could be $G$-injective).  Then $A\rtimes_{\max}G=A\rtimes_{\inj} G$.  \qed
\end{lemma}

The following lemma is well-known.  The proof is closely related to the proof of \cite[Theorem 4.9, (5) $\Rightarrow$ (6)]{Buss:2014aa}.

\begin{lemma}\label{int lem}
Let $A$ be a $G$-algebra, and $(\sigma,u):(A,G)\to \Bd(H)$ be a pair consisting of a ccp map $\sigma$ and a unitary representation $u$ satisfying the usual covariance relation 
$$
\sigma(\alpha_g(a))=u_g\sigma(a)u_g^*
$$
for all $a\in A$ and $g\in G$. Then the integrated form 
$$
\sigma\rtimes u :C_c(G,A)\to \Bd(H),\quad f\mapsto \sum_{g\in G} \sigma(f(g))u_g
$$
extends to a ccp map $\sigma\rtimes u:A\rtimes_{\max} G \to\Bd(H)$.
\end{lemma}

\begin{proof}
Replacing $A$ with its unitization and $\sigma$ with the canonical ucp extension to the unitization (\cite[Proposition 2.2.1]{Brown:2008qy}) we may assume that $A$ and $\sigma$ are unital.  Equip the algebraic tensor product $A\odot H$ with the inner product defined on elementary tensors by 
$$
\langle a\otimes \xi,b\otimes \eta\rangle:=\langle \xi,\sigma(a^*b)\eta\rangle.
$$
As in proofs of the usual Stinespring construction (see for example \cite[Proposition 1.5.1]{Brown:2008qy}), the fact that $\sigma$ is completely positive implies that this inner product is positive semi-definite, so we may take the separated completion to get a Hilbert space $H'$.  Let $\alpha$ denote the action of $G$ on $A$, let $g\in G$, and provisionally define a map $v_g:H'\to H'$ by the formula
$$
v_g:a\otimes \xi\mapsto \alpha_g(a)\otimes u_g\xi.
$$
Equivariance of $\phi$ implies that this preserves the inner product defined above, so each $v_g$ is well-defined and unitary.  It is then straightforward to check that $v$ defines a unitary representation $v$ of $G$ on $H'$.  Moreover, as in the usual Stinespring construction, for $a\in A$ the map $\widetilde{\sigma}(a)$ defined on elementary tensors by 
$$
\widetilde{\sigma}(a):b\otimes \xi\mapsto ab\otimes \xi
$$
gives a well-defined bounded operator on $H'$, and this defines a representation $\widetilde{\sigma}:A\to \Bd(H')$, which is covariant for the representation $v$ of $G$.  Again analogously to the usual Stinespring construction, 
$$
V:\xi\to 1_A\otimes \xi
$$
defines an equivariant isometry $V:H\to H'$ such that $V^*\widetilde{\sigma}(a)V=\sigma(a)$ for all $a\in A$.  One can now check that if 
$$
\widetilde{\sigma}\rtimes v:A\rtimes_{\max} G \to \Bd(H') 
$$
is the integrated form of the pair $(\widetilde{\sigma},v)$, then the map defined by
\begin{equation}\label{ccp ext}
A\rtimes_{\max} G \to \Bd(H),\quad a\mapsto V^*(\widetilde{\sigma}\rtimes v(a))V
\end{equation}
is an extension of the map 
$$
\sigma\rtimes u :C_c(G,A)\to \Bd(H),\quad f\mapsto \sum_G \sigma(f(g))u_g
$$
from the statement.  As the map defined in line \eqref{ccp ext} is clearly ccp, we are done. 
\end{proof}

The statement and proof of the following result are inspired by Lance's \emph{tensor product trick}: see for example the exposition in \cite[Proposition 3.6.6]{Brown:2008qy}, or the original article \cite{Lance:1973aa}.

\begin{theorem}\label{inj=max}
Let $\iota\colon A\into B$ be a faithful $G$-embedding. The following are equivalent:
\begin{enumerate}[(i)]
\item $\iota\rtimes_\max G\colon A\rtimes_{\max}G\to B\rtimes_{\max}G$ is injective;
\item for any covariant representation $(\pi,u)\colon (A,G)\to \Bd(H)$, there is a ccp $G$-map $\varphi\colon B\to \Bd(H)$ with $\varphi\circ\iota=\pi$;
\item there exists a covariant representation $(\pi,u):(A,G)\to \Bd(H)$ such that the integrated form $\pi\rtimes u:A\rtimes_{\max}G \to\Bd(H)$ is faithful and for which there is a ccp $G$-map $\varphi\colon B\to \Bd(H)$ with $\varphi\circ\iota=\pi$.
\end{enumerate}
\end{theorem}

\begin{proof}
Assume (i), and let $(\pi,u):(A,G)\to \Bd(H)$ be a covariant representation.  We must show that the dashed arrow below can be filled in with a ccp $G$-map
\begin{equation}\label{desideratum}
\xymatrix{ B \ar@{-->}[dr] & \\ A \ar[u]^-\iota \ar[r]^-\pi & \Bd(H) }.
\end{equation}
Let $\widetilde{A}$ and $\widetilde{B}$ be the unitzations of $A$ and $B$ and let $\widetilde{\pi}:\widetilde{A}\to \Bd(H)$ and $\widetilde{\iota}:\widetilde{A}\to \widetilde{B}$ be the canonical (equivariant) unital extensions.  It will suffice to prove that the dashed arrow below 
$$
\xymatrix{ \widetilde{B} \ar@{-->}[dr] & \\ \widetilde{A} \ar[u]^-{\widetilde{\iota}} \ar[r]^-{\widetilde{\pi}} & \Bd(H) }
$$
can be filled in with an equivariant ccp map; indeed, if we can do this, then the restriction of the resulting equivariant ccp map $\widetilde{B}\to \Bd(H)$ to $B$ will have the desired property.

Since the descent $\iota\rtimes G:A\rtimes_{\max} G \to B\rtimes_{\max} G$ of $\iota$ is injective by assumption,  it follows from this and the commutative diagram
$$
\xymatrix{ 0\ar[r] & B\rtimes_{\max} G \ar[r] & \widetilde{B}\rtimes_{\max} G \ar[r] & \C\rtimes_{\max} G \ar[r] & 0 \\
0\ar[r] & A\rtimes_{\max} G \ar[r] \ar[u] & \widetilde{A}\rtimes_{\max} G \ar[r] \ar[u] & \C\rtimes_{\max} G \ar@{=}[u] \ar[r] & 0 }
$$
of short exact sequences that the map 
$$
\widetilde{\iota}\rtimes G:\widetilde{A}\rtimes_{\max} G\to \widetilde{B}\rtimes_{\max} G
$$
is injective as well.  From now on, to avoid cluttered notation, we will assume that $A$, $B$, $\pi$ and $\iota$ are unital, and that the map $\iota\rtimes G:A\rtimes_{\max}G\to B\rtimes_{\max}G$  is injective; our goal is to fill in the dashed arrow in line \eqref{desideratum} under these new assumptions.

Now, as we are assuming that $\iota\rtimes G:A\rtimes_{\max} G \to B\rtimes_{\max} G$ is injective, in the diagram below
$$
\xymatrix{ B\rtimes_{\max} G \ar@{-->}[dr]^-{\widetilde{\pi\rtimes u}} & \\ A\rtimes_{\max}G \ar[r]^-{\pi \rtimes u} \ar[u]^-{\iota\rtimes G} & \Bd(H) }
$$
we may thus use injectivity of $\Bd(H)$ (i.e.\ Arveson's extension theorem as in for example \cite[Theorem 1.6.1]{Brown:2008qy}) to show that the dashed arrow can be filled in with a ucp map.  Any operator of the form $u_g$ is in the multiplicative domain of $\widetilde{\pi\rtimes u}$, from which it follows that the restriction $\phi$ of $\widetilde{\pi\rtimes u}$ to $B$ is equivariant.  This restriction $\phi$ is the desired map.

The implication (ii)$\Rightarrow$(iii) is clear, so it remains to show (iii)$\Rightarrow$(i).  Let $\pi\rtimes u:A\rtimes_{\max} G \to \Bd(H)$ be a faithful representation such that there is an ccp $G$-map $\widetilde{\pi}:B\to \Bd(H)$ that extends $\pi$ as in (iii).  Lemma \ref{int lem} implies that this ccp map integrates to a ccp map $\widetilde{\pi}\rtimes u:B\rtimes_{\max} G \to \Bd(H)$.  As the diagram 
$$
\xymatrix{ B\rtimes_{\max} G \ar[dr]^{\widetilde{\pi}\rtimes u} & \\ A\rtimes_{\max} G \ar[r]^-{\pi\rtimes u} \ar[u]^-{\iota\rtimes G} & \Bd(H) }
$$
commutes and the horizontal map is injective, the vertical map is injective too.
\end{proof}

Notice that $A\rtimes_\max G=A\rtimes_\inj G$ if and only if every $G$-embedding $\iota\colon A\into B$ satisfies the equivalent conditions in Proposition~\ref{inj=max}. Hence we get the following immediate consequence, for which we need one additional definition.

\begin{definition}\label{g inj rep}
A covariant representation $(\pi,u):(A,G)\to \Bd(H)$ is \emph{$G$-injective} if for any $G$-embedding $A\subseteq  B$ there exists a ccp map $\sigma:B\to \Bd(H)$ that extends $\pi$, and satisfies the covariance relation for $u$.
\end{definition}

\begin{corollary}\label{inj lem}
For a $G$-algebra $A$, the following are equivalent:
\begin{enumerate}[(i)]
\item $A\rtimes_\max G=A\rtimes_\inj G$;
\item every covariant representation $(\pi,u)$ is $G$-injective;
\item there is a $G$-injective covariant representation that integrates to a faithful representation of $A\rtimes_\max G$.
\end{enumerate}
Moreover, if $G$ is exact, $\rtimes_{\inj}$ may be replaced by $\rtimes_{\red}$ in the above. \qed
\end{corollary}

Our next goal is to develop this to get a characterization in terms of an amenability property of a more traditional `approximation property' form.

\begin{lemma}\label{unital lem}
Let $A$ be a $G$-algebra and let $(\pi,u)\colon (A,G)\to \Bd(H)$ be a nondegenerate $G$-injective covariant pair. Then for any unital $G$-algebra $C$ there exists a ucp $G$-map $\phi:C\to \pi(A)'\sbe \Bd(H)$.
\end{lemma}

\begin{proof}
Consider the canonical $G$-embedding 
$$
\iota\colon A\into C\otimes A,\quad a\mapsto 1\otimes a.
$$
Then $G$-injectivity of $\pi$ yields a ccp $G$-map $\varphi\colon C\otimes A\to \Bd(H)$ with $\varphi\circ\iota=\pi$.
Since $\pi$ is nondegenerate, so is $\varphi$, that is, $\varphi(e_i)\to 1$ strongly if $(e_i)$ is an approximate unit for $A$. 
It follows that $\varphi$ extends to a ucp map $\bar\varphi\colon \M(C\otimes A)\to\Bd(H)$, see \cite[Corollary~5.7]{Lance:1995ys}. Moreover, this extension is $G$-equivariant as can be seen from the construction of $\bar\varphi$ in \cite{Lance:1995ys}.  We now consider the canonical $G$-embedding $j\colon C\to \M(C\otimes A)$, $c\mapsto c\otimes 1$, and then define $\phi\colon C\to \Bd(H)$ by $\phi(c):= \bar\varphi(j(c))$. It remains to show that $\phi(C)\sbe \pi(A)'$. But since $\bar\varphi\circ\iota=\varphi\circ\iota=\pi$ is a homomorphism, the image of $\iota$ lies in the multiplicative domain of $\bar\varphi$, so that
\begin{multline*}
\phi(c)\pi(a)=\bar\varphi(j(c))\varphi(\iota(a))
=\bar\varphi(j(c)\iota(a))\\=\bar\varphi(\iota(a)j(c))=\varphi(\iota(a))\bar\varphi(j(c))=\pi(a)\phi(c).
\end{multline*}
\end{proof}

Here is the version of amenability we will use. To state it, if $A$ is a $G$-algebra and $(\pi,u):(A,G)\to \Bd(H)$ a covariant pair, then $\pi(A)'$ will 
be  equipped with the $G$-action $\beta=\Ad_u$ defined  by conjugation by $u$.

\begin{definition}\label{com amen}
Let $A$ be a $G$-algebra, and $(\pi,u):(A,G)\to \Bd(H)$ a covariant pair.  The pair $(\pi,u)$ is \emph{commutant amenable} (C-amenable) if there exists a net $(\theta_i:G\to \pi(A)')$ of positive type functions (with respect to $\beta=\Ad_u$) such that:
\begin{enumerate}[(i)]
\item each $\theta_i$ is finitely supported;
\item for each $i$, $\theta_i(e)\leq 1$;
\item for each $g\in G$, $\theta_i(g)\to 1$ ultraweakly as $i\to\infty$.
\end{enumerate}
The $G$-algebra $A$ is \emph{commutant amenable} (C-amenable) if every covariant pair is C-amenable.
\end{definition}

\begin{remark}\label{ca rem}
If a $G$-algebra $A$ is amenable, then it is C-amenable.  This follows as any covariant representation $(\pi,u)$ of $(A,G)$ extends to a covariant representation of $(A^{**},G)$, and as the image of $Z(A^{**})$ under this extension is necessarily contained in the commutant $\pi(A)'$.
\end{remark}

We give the above definition to make the analogy with amenability clearer. 
However, it will be more convenient to work with the following reformulation. To state it, recall that if $(\pi,u)\colon (A,G)\to \Bd(H)$
is a covariant representation of the $G$-algebra $A$, we use the action $\beta=\Ad_u$ on the commutant $\pi(A)'$ 
to define $\ell^2(G,\pi(A)')$ as in Definition \ref{standard mod}.  The proof of the next lemma is essentially the same as that of Lemma \ref{l2 amen}, and so omitted.

\begin{lemma}\label{ca reform}
Let $A$ be a $G$-algebra with action $\alpha$, and let $(\pi,u)$ be a covariant pair.  Then $(\pi,u)$ is C-amenable if and only if there exists a net $(\xi_i)$ in $\ell^2(G,\pi(A)')$ such that:
\begin{enumerate}
\item each $\xi_i$ is finitely supported;
\item for each $i$, $\langle \xi_i,\xi_i\rangle \leq 1$;
\item for each $g\in G$, $\langle \xi_i,\widetilde{\beta}_g(\xi_i)\rangle \to 1$ ultraweakly as $i\to\infty$. \qed
\end{enumerate}
\end{lemma}

\begin{proposition}\label{pro:CA=>max=red}
Let $A$ be a $G$-algebra, and say there exists a C-amenable covariant pair $(\pi,u)$ which integrates to a faithful representation of $A\rtimes_{\max}G$. Then $A\rtimes_\max G=A\rtimes_\red G$.
\end{proposition}

\begin{proof}
Let $(\pi,u):(A,G)\to \Bd(H)$ be a covariant pair as in the statement.  Let $(\xi_i:G\to \pi(A)')$ be a net as in Lemma \ref{ca reform}.  For each $i$, define 
$$
T_i:H\to \ell^2(G,H),\quad v\mapsto \big(g\mapsto \xi_i(g)v\big) .
$$
A direct computation shows that  
\begin{align*}
\|T_iv\|^2 & =\langle v,\langle \xi_i,\xi_i\rangle v\rangle.
\end{align*}
As $\langle \xi_i,\xi_i\rangle \leq 1$ for all $i$, the term on the right is bounded above by $\|v\|^2$, and thus $\|T_i\|\leq 1$.  The adjoint of $T_i$ is easily seen to be given by 
$$
T_i^*(\eta)=\sum_G\xi_i(g)^*\eta(g) 
$$ 
for all $\eta\in \contc(G,H)$.

Now, via Fell's trick, the covariant pair $(\pi\otimes 1,u\otimes \lambda):(A,G)\to \Bd(\ell^2(G,H))$ integrates to $A\rtimes_{\red} G$.  Consider now the ccp map
$$
\phi_i:\Bd(\ell^2(G,H))\to\Bd(H),\quad b\mapsto T_i^*bT_i.
$$
We compute for $f\in \contc(G,A)$ and $v\in H$:
\begin{align*}
\phi_i\circ (\pi\otimes 1)\rtimes (u\otimes \lambda)(f)&=T_i^*((\pi\otimes 1)\rtimes (u\otimes \lambda)(f))T_i\\
&=\sum_{g,h\in G}\xi_i(h)^*\pi(f(g))u_g\xi_i(g^{-1}h)
\end{align*}
Using that $\xi_i$ takes values in $\pi(A)'$, this equals
\begin{align*}
\sum_{g\in G}\pi(f(g))\left(\sum_{h\in G}\xi_i(h)^*u_g\xi_i(g^{-1}h)u_g^*\right)u_g
=\sum_G\pi(f(g))\langle \xi_i,\widetilde{\beta}_{g}\xi_i\rangle u_g.
\end{align*}
As $\langle \xi_i,\widetilde{\beta}_{g}\xi_i\rangle$ converges ultraweakly to $1$ and as multiplication is separately ultraweakly continuous, we get ultraweak convergence 
$$
\phi_i\circ (\pi\otimes 1)\rtimes (u\otimes \lambda)(f)\to (\pi\rtimes u)(f) \text{ as } i\to\infty.
$$
As ultraweak limits do not increase norms and as each $\phi_i$ is ccp, we get
$$
\|(\pi\rtimes u)(f)\|\leq \limsup_{i\to\infty}\|\phi_i\circ (\pi\otimes 1)\rtimes (u\otimes \lambda)(f)\|\leq \|(\pi\otimes 1)\rtimes (u\otimes \lambda)(f)\|.
$$
Hence as $(\pi\otimes 1)\rtimes (u\otimes \lambda)$ extends to $A\rtimes_{\red} G$, we get 
$$
\|(\pi\rtimes u)(f)\|\leq \|f\|_{A\rtimes_{\red} G}.
$$
As $\pi\rtimes u$ is faithful on $A\rtimes_{\max}G$, however, we are done.
\end{proof}

Finally in this section, we are able to give a characterization of weak containment in terms of commutant amenability, at least for exact groups.

\begin{theorem}\label{com amen the}
Let $G$ be an exact discrete group, and let $A$ be a $G$-$C^*$-algebra.  Then the following are equivalent:
\begin{enumerate}[(i)]
\item $A$ is commutant amenable;
\item $A\rtimes_{\max}G=A\rtimes_{\red} G$;
\item $A\rtimes_{\max}G=A\rtimes_{\inj} G$.
\end{enumerate}
\end{theorem}

\begin{proof}
The implication (i)$\Rightarrow$(ii) is Proposition~\ref{pro:CA=>max=red}.
The implication (ii)$\Rightarrow$(iii) is trivial, so it remains to show that $A\rtimes_{\max}G=A\rtimes_{\inj} G$ implies C-amenability.  Let then $(\pi,u)$ be a covariant pair for $(A,G)$.  We may apply Corollary \ref{inj lem} and Lemma \ref{unital lem} to get an equivariant ucp map $\phi:\ell^\infty(G)\to \pi(A)'$.  As $G$ is exact, $\ell^\infty(G)$ is strongly amenable; postcomposing a net $(\theta_i:G\to \ell^\infty(G))$ that shows $\ell^\infty(G)$ is strongly amenable  with $\phi$ gives C-amenability.
\end{proof}

\begin{remark} Similar to Remark \ref{rem-norm}, it follows from the above proof that for $G$ exact we can replace ultraweak convergence 
of the net $\theta_i:A\to \pi(A)'$ in the definition of C-amenability by norm convergence.
\end{remark}

\section{Matsumura's characterisations of weak containment}\label{comm case sec}

In this section, we connect the ideas in the previous section to other forms of amenability, and related results.  The key ideas here are due to Matsumura \cite{Matsumura:2012aa}, and the results are essentially fairly mild generalizations of Matsumura's.  Nonetheless, our proofs are somewhat different from those of \cite{Matsumura:2012aa}.  We also think some of the generalizations are worthwhile in their own right: for example, we remove some unitality and nuclearity assumptions, and have some applications to actions of non-exact groups.

The key technical tool in this section is a seminal theorem of Haagerup \cite{Haagerup:1975xh} on the existence of standard forms.  We summarize what will be the key points for us in the next theorem (see \cite{Matsumura:2012aa}*{Theorem 2.3} for a brief discussion how the following theorem follows from \cite{Haagerup:1975xh}).

\begin{theorem}\label{hsf the}
Let $A$ be a $G$-algebra.  There exist \emph{standard form representations} 
$$
\pi:A^{**}\to \Bd(H),\quad\text{and} \quad \pi^{\op}:(A^{\op})^{**}\to \Bd(H)
$$
on the same Hilbert space $H$
together with a unitary representation $u:G\to \U(H)$
with the following properties:
\begin{enumerate}[(i)]
\item $\pi$ and $\pi^{\op}$ are normal, unital, and faithful;
\item $(\pi, u)$ and $(\pi^{\op}, u)$ are covariant with respect to the canonical $G$-actions on $A^{**}$ and $(A^{\op})^{**}$;
\item having identified $A^{**}$ and $(A^{\op})^{**}$ with their images under $\pi$ and $\pi^{\op}$, we get $\pi(A)'=(A^{\op})^{**}$ and $\pi^{\op}(A^{\op})'=A^{**}$;
\item if $A$ is commutative, then $\pi(A)'=A^{**}$.  
\end{enumerate}
\end{theorem}

The cleanest results we can prove on weak containment are in the case when $G$ is exact and the $G$-algebra $A$ is commutative, so we turn to this first.

\begin{theorem}\label{com ex the}
Let $G$ be an exact group, and $A$ a commutative $G$-algebra.  The following are equivalent: 
\begin{enumerate}[(i)]
\item \label{com boa} $A$ is strongly amenable;
\item \label{com ada} $A$ is amenable;
\item \label{com ca} $A$ is C-amenable;
\item \label{com m=r} $A\rtimes_{\max}G=A\rtimes_{\red}G$;
\item \label{com m=i} $A\rtimes_{\max}G=A\rtimes_{\inj}G$;
\item \label{com ** boa} $A^{**}$ is strongly amenable;
\item \label{com ** ada} $A^{**}$ is amenable;
\item \label{com ** ca} $A^{**}$ is C-amenable;
\item \label{com ** m=r} $A^{**}\rtimes_{\max}G=A^{**}\rtimes_{\red} G$;
\item \label{com ** m=i} $A^{**}\rtimes_{\max}G=A^{**}\rtimes_{\inj} G$.
\end{enumerate}
\end{theorem}

\begin{proof}
That \eqref{com boa} implies \eqref{com ada} is trivial, and that \eqref{com ada} implies \eqref{com ca} follows from Remark \ref{ca rem}.  That \eqref{com ca} implies \eqref{com m=r} is Proposition \ref{pro:CA=>max=red}, and \eqref{com m=r} implies \eqref{com m=i} is immediate.  Assume condition \eqref{com m=i}; we will show condition \eqref{com ** boa}.  Indeed, let $\pi:A\to\Bd(H)$ be the restriction of a standard form of $A^{**}$ as in Theorem \ref{hsf the} to $A$, so $\pi$ is covariant for some unitary representation $u$ on $H$, and $\pi(A)'$ identifies naturally with $A^{**}$.  Corollary \ref{inj lem} and Lemma \ref{unital lem}  give us an equivariant ucp map
$$
\phi:\ell^\infty(G)\to A^{**}.
$$
It follows from Proposition \ref{G-exact-char-amenable} and exactness of $G$ that $A^{**}$ is strongly amenable.

Continuing, \eqref{com ** boa} implies \eqref{com ** ada} is trivial and \eqref{com ** ada} implies \eqref{com ** ca} is Remark \ref{ca rem} again, while \eqref{com ** ca} implies \eqref{com ** m=r} is Proposition \ref{pro:CA=>max=red} again.  Also, \eqref{com ** m=r} implies \eqref{com ** m=i} is again immediate.  Finally, it remains to show that \eqref{com ** m=i} implies \eqref{com boa}.  For this, note that Lemma \ref{unital lem} applies to a 
standard form representation $(\pi,u)$
to give us an equivariant ucp map 
$$
\phi:\ell^\infty(G)\to A^{**}.
$$

Using Proposition \ref{G-exact-char-amenable} and exactness again, this implies that $A$ is amenable.  The proof is completed by appealing to \cite[Th\'{e}or\`{e}me 4.9]{Anantharaman-Delaroche:1987os} (and commutativity of $A$) to get from there back to strong amenability of $A$.
\end{proof}

Many of the equivalences of Theorem \ref{com ex the} were known before this paper: notably, Matsumura \cite{Matsumura:2012aa} showed that \eqref{com m=r} and \eqref{com ada} are equivalent when $A$ is unital, and the equivalence of \eqref{com boa} and \eqref{com ada} (and similarly of \eqref{com ** boa} and \eqref{com ** ada})  is due to Anantharaman-Delaroche \cite[Th\'{e}or\`{e}me 4.9]{Anantharaman-Delaroche:1987os}.  
Many of the other implications are either trivial, or probably known to at least some experts, so overall we certainly don't claim  much profundity!  

Nonetheless, we hope collecting these conditions in one place clarifies the theory somewhat.  It is perhaps also interesting that the only real way the proof goes beyond the classical results of Anantharaman-Delaroche from her seminal 1987 paper \cite{Anantharaman-Delaroche:1987os} are in the existence of standard forms \cite{Haagerup:1975xh} from 1979 and the adaptations in Corollary \ref{inj lem} and Lemma \ref{unital lem} of well-known tricks involving injectivity and multiplicative domains due originally to Lance \cite{Lance:1973aa} (see also the exposition in \cite[Section 3.6]{Brown:2008qy}), and dating to 1973.  Some of this already appears in Matsumura's work.

In the case where $G$ is not necessarily exact, some of Theorem \ref{com ex the} still holds; however, some of it becomes false, and other parts are unclear.

\begin{theorem}\label{com nex 1}
Let $G$ be a discrete group, and let $A$ be a commutative $G$-algebra.  Consider the following conditions:
\begin{enumerate}[(i)]
\item \label{com nex ** boa} $A^{**}$ is strongly amenable;
\item \label{com nex ** ada} $A^{**}$ is amenable;
\item \label{com nex ** c} $A^{**}$ is C-amenable;
\item \label{com nex ** m=r} $A^{**}\rtimes_{\max}G=A^{**}\rtimes_{\red} G$;
\item \label{com nex ** m=i} $A^{**}\rtimes_{\max}G=A^{**}\rtimes_{\inj} G$;
\item \label{com nex boa} $A$ is strongly amenable;
\item \label{com nex ada} $A$ is amenable;
\item \label{com nex c} $A$ is C-amenable.
\end{enumerate}
We have in general that 
\begin{equation}\label{strong nex}
\eqref{com nex ** boa} \Leftrightarrow \eqref{com nex ** ada} \Leftrightarrow \eqref{com nex ** c},
\end{equation}
that 
\begin{equation}\label{weak nex}
\eqref{com nex ** m=i} \Leftrightarrow \eqref{com nex boa} \Leftrightarrow \eqref{com nex ada} \Leftrightarrow \eqref{com nex c}
\end{equation}
and that the conditions in line \eqref{strong nex} imply condition \eqref{com nex ** m=r}, which in turn implies the conditions in line \eqref{weak nex}.

Moreover, if $A$ is unital, all the conditions listed above are equivalent, and if $A\not=0$, they force $G$ to be exact.
\end{theorem}

\begin{proof}
The implication \eqref{com nex ** boa} $\Rightarrow$ \eqref{com nex ** ada} is trivial, and the converse is part of \cite[Th\'{e}or\`{e}me 4.9]{Anantharaman-Delaroche:1987os}.  The implication \eqref{com nex ** ada} $\Rightarrow$ \eqref{com nex ** c} is likewise trivial, and the converse is a consequence of C-amenability of a standard form representation (see Theorem \ref{hsf the}), and the fact that the bidual of a commutative C*-algebra is again commutative.  Thus we have the equivalences in line \eqref{strong nex}.

The implication \eqref{com nex ** c} $\Rightarrow$ \eqref{com nex ** m=r} is Proposition \ref{pro:CA=>max=red}, and \eqref{com nex ** m=r} $\Rightarrow$ \eqref{com nex ** m=i} is trivial.   

Look finally at the equivalences in line \eqref{weak nex}.  First note that condition \eqref{com nex ** m=i} implies via Corollary \ref{inj lem} that a standard form representation (Theorem \ref{hsf the}) $\pi$ of $A^{**}$ is an injective representation.  In particular, we can fill in the dashed line below with a ucp $G$-map making the diagram commute
$$
\xymatrix{ \ell^\infty(G,A^{**}) \ar@{-->}[dr]^-{\widetilde{\pi}} & \\ A^{**} \ar[u] \ar[r]^-\pi & \Bd(H) }
$$
where the vertical arrow is the canonical inclusion of $A^{**}$ in $\ell^\infty(G,A^{**})$ as constant functions.  As $A^{**}$ is in the multiplicative domain of $\widetilde{\pi}$, it follows that $\widetilde{\pi}(\ell^\infty(G,A^{**}))$ commutes with $\pi(A^{**})$; however, as $\pi$ is a standard form, this implies that the image of $\widetilde{\pi}$ is exactly $\pi(A^{**})\cong A^{**}$.  In other words, $\widetilde{\pi}$ identifies with an equivariant conditional expectation $\ell^\infty(G,A^{**})\to A^{**}$ that splits the constant inclusion.  The existence of such a conditional expectation is equivalent to amenability by \cite[Th\'{e}or\`{e}me 3.3]{Anantharaman-Delaroche:1987os}.  Conversely, assuming amenability, we have such an equivariant conditional expectation.  As $A$ is commutative, $A^{**}$ is injective, whence $\ell^\infty(G,A^{**})$ is $G$-injective by Example \ref{inj ex}.  Hence $A^{**}$ is also $G$-injective, whence we get condition \eqref{com nex ** m=i} by Lemma \ref{g wep lem}.

To complete the equivalences in line \eqref{weak nex}, note that the implications from \eqref{com nex boa} to \eqref{com nex ada}, and from \eqref{com nex ada} to \eqref{com nex c} are straightforward.  On the other hand, we have implications from \eqref{com nex c} to \eqref{com nex ada} by C-amenability of a standard form representation, and from \eqref{com nex ada} to \eqref{com nex boa} by part of \cite[Th\'{e}or\`{e}me 4.9]{Anantharaman-Delaroche:1987os}.

In the unital case, the implication from \eqref{com nex boa} to \eqref{com nex ** boa} is trivial, so we are done.
\end{proof}

We also get the following observation in general.

\begin{theorem}\label{com nex 2} 
Let $A$ be a commutative $G$-algebra.  The following are equivalent:
\begin{enumerate}[(i)]
\item \label{com nex m=i} $A\rtimes_{\max}G=A\rtimes_{\inj}G$;
\item \label{com nex wep} $A$ has the $G$-WEP.
\end{enumerate}
\end{theorem}

\begin{proof}
The implication from \eqref{com nex m=i} to \eqref{com nex wep} follows as if $A\to \ell^\infty(G,A^{**})$ is the equivariant inclusion of $A$ as constant functions from $G$ to $A$ and $\pi:A\to A^{**}\subseteq \Bd(H)$ is the restriction of a standard form of $A^{**}$ from Theorem \ref{hsf the} to $A$, then Corollary \ref{inj lem} gives us a commutative diagram
$$
\xymatrix{ \ell^\infty(G,A^{**}) \ar[dr]^-\phi & \\ A \ar[u]\ar[r]^-\pi & \Bd(H) }
$$
with $\phi$ a ccp $G$-map.  As $A$ is in the multiplicative domain of $\phi$ and the algebra $\ell^\infty(G,A^{**})$ is commutative, $\phi$ takes image in $\pi(A)'=A^{**}$.  In other words we have factored the canonical inclusion $A\into A^{**}$ through the $G$-injective algebra (Example \ref{inj ex}) $\ell^\infty(G,A^{**})$, which easily implies the $G$-WEP.

The converse is Lemma \ref{g wep lem}.
\end{proof}

\begin{remark}\label{com nex 2 rem}
Any of the conditions in Theorem \ref{com nex 1} imply that $A\rtimes_{\max}G=A\rtimes_r G$: indeed, they all apply C-amenability of $A$ by that theorem, and this implies $A\rtimes_{\max}G=A\rtimes_r G$ by Proposition \ref{pro:CA=>max=red}.  Moreover, the condition that $A\rtimes_{\max}G=A\rtimes_r G$ trivially implies $A\rtimes_{\max}G=A\rtimes_{\inj}G$ for any $G$-algebra $A$.  Summarizing, 
$$
\text{(Theorem \ref{com nex 1} conditions)}~\Rightarrow ~A\rtimes_{\max}G=A\rtimes_r G~\Rightarrow ~\text{(Theorem \ref{com nex 2} conditions)}.
$$
On the other hand, the conditions in Theorem \ref{com nex 2} are true for $A=\ell^\infty(G)$ and any $G$, while those for Theorem \ref{com nex 1} are all false for any non-exact $G$ and $A=\ell^\infty(G)$, so we do not have equivalence of the conditions in Theorems \ref{com nex 1} and \ref{com nex 2}.  The exact relationship between the conditions in Theorem \ref{com nex 1} and \ref{com nex 2} and the condition $A\rtimes_{\max}G=A\rtimes_r G$ is not at all clear in general.
\end{remark}

\begin{remark}\label{com nex 1 rem}
The conditions in Theorem \ref{com nex 1} are not all equivalent in the non-exact case without the assumption that $A$ is unital: $A=C_0(G)$ satisfies the conditions in line \eqref{weak nex} for any $G$, but never satisfies the conditions in line \eqref{strong nex} when $G$ is not exact.  The $G$-algebra $A=C_0(G)$ is also known to fail \eqref{com nex ** m=r} for at least some non-exact groups (see \cite[Lemma 4.7]{Buss:2018nm}), so \eqref{com nex ** m=r} cannot always be equivalent to the conditions in line \eqref{weak nex}.  It is possible that condition \eqref{com nex ** m=r} is always equivalent to those in line \eqref{strong nex} even for non-unital $A$; however, it also seems quite plausible that \eqref{com nex ** m=r} holds for some non-exact groups and $A=C_0(G)$.  This issue seems quite open at the moment.
\end{remark}

\begin{remark}\label{com nex 1 rem 2}
Note that condition \eqref{com nex ** boa} in Theorem \ref{com nex 1} is never satisfied for any $A$ when $G$ is not exact: indeed, it implies that the action of $G$ on the spectrum of $A^{**}$, a compact space, is amenable, which is well-known to imply exactness.  Hence none of the conditions in Theorem \ref{com nex 1} can be satisfied when $A$ is unital and $G$ is not exact; the theorem still seems somewhat interesting in this case as simply saying that none of the conditions are possible.  The bottom four conditions are possible for non-exact $G$ and non-unital $A$: all happen for $A=C_0(G)$.  It is not clear whether or not condition \eqref{com nex ** m=r} can ever happen for non-exact $G$ and non-unital $A$: as already remarked, it seems plausible that this can happen for some $G$ and $A=C_0(G)$.
\end{remark}

We now move on to the noncommutative case.  

\begin{theorem}\label{nc ex the}
Let $G$ be exact and $A$ a $G$-algebra.  Let $A^{\op}$ be the opposite algebra of $A$, and let $M$ be the von Neumann algebra generated by $\pi(A)$ and $\pi^{\op}(A^{\op})$ in standard forms $\pi$ and $\pi^{\op}$ of $A^{**}$ and $(A^{\op})^{**}$ (see Theorem \ref{hsf the}).  The following are equivalent:
\begin{enumerate}[(i)]
\item \label{nc ada} $A$ is amenable ;
\item \label{nc c} $A\otimes_{\max} A^{\op}$ is C-amenable;
\item \label{nc m=r} $(A\otimes_{\max}A^{\op})\rtimes_{\max}G=(A\otimes_{\max}A^{\op})\rtimes_{\red}G$;
\item \label{nc m=i} $(A\otimes_{\max}A^{\op})\rtimes_{\max}G=(A\otimes_{\max}A^{\op})\rtimes_{\inj}G$;
\item \label{nc ** boa} $A^{**}$ is strongly amenable;
\item \label{nc m boa} $M$ is strongly amenable;
\item \label{nc m ada} $M$ is amenable;
\item \label{nc m c} $M$ is C-amenable;
\item \label{nc ** m=r} $M\rtimes_{\max}G=M\rtimes_{\red} G$;
\item \label{nc ** m=i} $M\rtimes_{\max}G=M\rtimes_{\inj} G$.
\end{enumerate}
\end{theorem}

\begin{proof}
Assume first that $A$ is amenable, let $B=A\otimes_{\max} A^{\op}$, and let $(\rho,v):(B,G)\to \Bd(H)$ be any covariant pair for $B$.  Then $\rho$ `restricts' to representations of $A$ and $A^{\op}$ as in \cite[Theorem 3.2.6]{Brown:2008qy}, which we also denote $\rho$.  Extending $\rho$ to $A^{**}$, we have that $\rho(Z(A^{**}))$ commutes with both $\rho(A)$ and $\rho(A^{\op})$, and therefore $\rho(Z(A^{**}))\subseteq \rho(B)'$.  It follows from this that amenability of $A$ implies C-amenability of $A\otimes_{\max}A^{\op}$, so we get \eqref{nc ada} implies \eqref{nc c}.  

The implication from \eqref{nc c} to \eqref{nc m=r} is Proposition \ref{pro:CA=>max=red}, and from \eqref{nc m=r} to \eqref{nc m=i} is trivial.  Assuming \eqref{nc m=i}, let $\pi:A\to\Bd(H)$ be the restriction of a standard form (see Theorem \ref{hsf the}) of $A^{**}$ to $A$.  Thanks to Theorem \ref{hsf the} and the universal property of the maximal tensor product we obtain the covariant representation $(\sigma=\pi\otimes_{\max} \pi^{\op}, u)$, of $A\otimes_{\max}A^{\op}$.  Lemma \ref{unital lem} now gives us a ucp $G$-map $\phi:\ell^\infty(G)\to \sigma(A\otimes_{\max}A^{\op})'$.  However, 
\begin{equation}\label{com op}
\sigma(A\otimes_{\max}A^{\op})'=\pi(A)'\cap \pi^{\op}(A)'=\pi(A)'\cap \pi(A)''=Z(\pi(A)')\cong Z(A^{**}).
\end{equation}
As $G$ is exact, the existence of an equivariant ucp $G$-map $\phi:\ell^\infty(G)\to Z(A^{**})$ implies that $A^{**}$ is strongly amenable  by Proposition \ref{G-exact-char-amenable}.  Hence we have shown that \eqref{nc m=i} implies \eqref{nc ** boa}.  

The implication from \eqref{nc ** boa} to \eqref{nc m boa} follows as a standard form $\pi:A^{**}\to \Bd(H)$ restricts to a unital equivariant $*$-homomorphism $\pi:Z(A^{**})\to Z(M)$.  The implications from \eqref{nc m boa} to \eqref{nc m ada} and \eqref{nc m ada} to \eqref{nc m c} are straightforward, and that from \eqref{nc m c} to \eqref{nc ** m=r} is Proposition \ref{pro:CA=>max=red} again.  

The implication from \eqref{nc ** m=r} to \eqref{nc ** m=i} is trivial, so it remains to get back from \eqref{nc ** m=i} to \eqref{nc ada}.  Indeed, Lemma \ref{unital lem} gives an equivariant ucp map $\phi:\ell^\infty(G)\to M'$, and analogously to line \eqref{com op}, $M'\cong Z(A^{**})$.  Proposition \ref{G-exact-char-amenable} again completes the proof.
\end{proof}

\begin{remark}\label{nc inj rem}
The conditions in Theorem \ref{nc ex the} are not equivalent to strong amenability of $A$.  This follows from the properties of Suzuki's examples \cite{Suzuki:2018qo} as discussed in Section \ref{suzuki sec}.
\end{remark}

When $G$ is not necessarily exact and $A$ is a general $C^*$-algebra, rather little of Theorems \ref{com nex 1} and Theorems \ref{com nex 2} seem directly recoverable: some implications do still hold, of course, as the interested reader can extract from the proof of Theorem \ref{nc ex the} above.  

Here are at least some implications that hold in general.

\begin{proposition}\label{ada vs c}
Let $A$ be a $G$-algebra.  The following are equivalent:
\begin{enumerate}[(i)]
\item \label{am} $A$ is amenable;
\item \label{tens con} $A\otimes_\max B$ is amenable  for every $G$-algebra $B$;
\item $A\otimes_\max A^{\op}$ is amenable;
\item $A\otimes_\max A^{\op}$ is C-amenable.
\end{enumerate}
\end{proposition}

\begin{proof}
The implication from (i) to (ii) is well-known, but maybe not explicit in the literature.
It can be proved using Lemma \ref{l2 amen} to get a net $(\xi_i\colon G\to Z(A^{**}))_{i\in I}$ with the properties states there, and the fact that there is a canonical $G$-embedding $A^{**}\into (A\otimes_\max B)^{**}$ inducing a $G$-embedding $Z(A^{**})\to Z((A\otimes_\max B)^{**})$.
The implication  from (ii) to (iii) is trivial, and (iii) implies (iv) is Remark \ref{ca rem}.  Finally, if (iv) holds, we use the same technique from the previous proofs by making use of the standard representation $\pi\colon A^{**}\into \Bd(H)$ to build a representation $$
\sigma:=\pi\otimes_\max \pi^{\op}\colon A\otimes_\max A^{\op}\to \Bd(H)
$$
with
$$
\sigma(A\otimes_\max A^\op)'=\pi(A)'\cap \pi(A)''\cong Z(A^{**}).
$$
Then C-amenability for $A\otimes_\max A^\op$ implies the existence of an approximative net $(\xi_i\colon G\to \rho(A\otimes_\max A^\op)'\cong Z(A^{**}))$ giving amenability for $A$.
\end{proof}

\begin{remark}\label{ca tens}
It would be interesting to know whether C-amenability passes to (maximal) tensor products in the sense of the implication \eqref{am} $\Rightarrow$  \eqref{tens con} from Proposition \ref{ada vs c} above, even for trivial actions.  Indeed, it would then follow if $A$ is a C-amenable and nuclear $G$-algebra, and $B$ is any $C^*$-algebra (with trivial $G$-action) then 
\begin{align*}
(A\rtimes_{r}G)\otimes_{\max} B &  = (A\rtimes_{\max}G)\otimes_{\max} B=(A\otimes_{\max} B)\rtimes_{\max} G = (A\otimes B)\rtimes_{\max} G  \\ & = (A\otimes B)\rtimes _r G =(A\rtimes_r G)\otimes B,
\end{align*}
where we have used Proposition \ref{pro:CA=>max=red} (twice), nuclearity of $A$ to replace the maximal tensor product with the spatial one, and standard facts about commuting tensor products with trivial $G$-algebras with crossed products.  This implies that $A\rtimes_r G$ is nuclear.  However, in \cite[Th\'{e}or\`{e}me 4.5]{Anantharaman-Delaroche:1987os}, Anantharaman-Delaroche proves that this is equivalent to amenability of $A$.  

To summarize, if we knew the implication 
\eqref{am} $\Rightarrow$  \eqref{tens con} of Proposition \ref{ada vs c} also held for C-amenability, we could conclude that C-amenability and amenability were equivalent for all actions on nuclear $C^*$-algebras (and therefore also that amenability was equivalent to $A\rtimes_{\max}G=A\rtimes_r G$ in the nuclear case).
\end{remark}

\section{Can non-exact groups admit amenable actions on unital $C^*$-algebras?}\label{exact sec}

It is well-known that a group admits a strongly amenable action on a unital commutative $G$-algebra if and only if the group is exact.  From this, it is clear that a non-exact group cannot admit a strongly amenable  action on a unital $G$-algebra: indeed, the action on the unital commutative $C^*$-algebra $Z(A)$ is then also strongly amenable.  For commutative $G$-algebras, strong amenability and amenability are equivalent by  \cite[Th\'{e}or\`{e}me 4.9]{Anantharaman-Delaroche:1987os}, and therefore a non-exact group cannot admit an amenable  action on a unital commutative $C^*$-algebra either.

However, as discussed in Section \ref{suzuki sec}, Suzuki's examples show that amenability and strong amenability are not equivalent for unital noncommutative (even nuclear) $G$-algebras.  It is therefore natural to ask whether a non-exact group can admit an amenable  action on any unital $G$-algebra.  

The purpose of this section is to show that the answer to this question is `no', even if we replace amenability by the a priori weaker condition of C-amenability.

\begin{theorem}\label{com amen ex}
Say $G$ is a discrete group and $A$ is a unital (nonzero) C-amenable $G$-algebra.  Then the following hold:
\begin{enumerate}[(i)]
\item if $A$ is nuclear, the inclusion 
$$
A\rtimes_{\red} G \to (A\otimes A^{\op})\rtimes_{\red} G
$$
induced by the equivariant map $A\to A\otimes A^{\op}$, $a\mapsto a\otimes 1$, is nuclear;
\item if $A$ is exact, then $A\rtimes_{\red} G$ is exact;
\item $G$ is exact.
\end{enumerate}
\end{theorem}

We have recalled in Theorem~\ref{ozawa the} that a discrete group $G$ is exact if and only if $\ell^\infty(G)$ is amenable.  Thus we have the following result.

\begin{corollary}
For a discrete group $G$, the following assertions are equivalent:
\begin{enumerate}[(i)]
\item $G$ is exact;
\item $G$ admits a strongly amenable  action on a unital nonzero \cstar{}algebra;
\item $G$ admits an amenable  action on a unital nonzero \cstar{}algebra;
\item $G$ admits a C-amenable action on a unital nonzero \cstar{}algebra.
\end{enumerate}
\end{corollary}

\begin{proof}
The implications (ii)$\Rightarrow$(iii)$\Rightarrow$(iv) are straightforward, and (iv)$\Rightarrow$(i) is Theorem~\ref{com amen ex}.  The implication from (i) to (ii) follows as if $G$ is exact, then $\ell^\infty(G)$ is strongly amenable (see Theorem~\ref{ozawa the}).
\end{proof}

For the proof of Theorem~\ref{com amen ex} we need some technical preparations.  For a $G$-algebra $A$, see Definition \ref{standard mod} for the module $\ell^2(G,A)$.  The proof of the following lemma is based partly on ideas of Anantharaman-Delaroche from the paper \cite{Anantharaman-Delaroche:1987os}.

\begin{lemma}\label{com amen}
Say $G$ is a discrete group, and $A$ is a unital  C-amenable $G$-$C^*$-algebra.  Then for any $\epsilon>0$ and any finite subset $\mathcal{G}$ of $G$ there exists a function $\xi\in \ell^2(G,A^{\op})$ such that:
\begin{enumerate}[(i)]
\item $\xi$ is finitely supported;
\item $\langle \xi,\xi\rangle \leq 1$;
\item for all $g\in \mathcal{G}$, $\|1_{A^{{\op}}}-\langle  \xi,\widetilde{\alpha}_g\xi\rangle\|<\epsilon$.
\end{enumerate}
\end{lemma}

\begin{proof}
Let $\pi:A\to \Bd(H)$ be the restriction of a standard form representation (see Theorem \ref{hsf the}) of $A^{**}$ to $A$.  Fix $\epsilon>0$, and finite subsets $\mathcal{G}$ of $G$, and $\Phi$ of the state space of $A^{\op}$ respectively.  As $A$ is commutant amenable, there exists a finitely supported function $\xi:G\to \pi(A)'$ such that $\langle \xi,\xi\rangle\leq  1$ and such that 
\begin{equation}\label{assume}
|\phi\big(\langle \xi,\widetilde{\alpha}_{g}\xi\rangle -1\big)|<\frac\epsilon 3
\end{equation}
for all $g\in \mathcal{G}$ and $\phi\in \Phi$.  As $\pi(A)'$ canonically identifies with $(A^{\op})^{**}$ we have from \cite[Lemme 1.1]{Anantharaman-Delaroche:1987os} that the finitely supported elements in the unit ball of $\ell^2(G,A^{\op})$ are dense in the unit ball of $\ell^2(G,\pi(A)')$ for the topology defined by the seminorms
$$
\|\eta\|_\psi:=\sqrt{\psi(\langle \eta,\eta\rangle)}
$$
as $\psi$ ranges over the state space of $A^{\op}$.  Hence there exists a finitely supported function $\eta:G\to A^{\op}$ such that $\langle \eta,\eta\rangle \leq 1$, and such that
\begin{equation}\label{str norm small}
\phi(\langle \xi-\eta, \xi-\eta\rangle)^{1/2}<\frac{\epsilon}{3} 
\end{equation}
for all $\phi\in \Phi$.  For each $\phi\in \Phi$ and $g\in \mathcal{G}$, we have 
\begin{equation}\label{start}
|\phi(\langle \eta,\widetilde{\alpha}_g\eta\rangle-\langle \xi,\widetilde{\alpha}_g\xi\rangle)|\leq |\phi(\langle \eta-\xi,\widetilde{\alpha}_g\eta\rangle)|+|\phi(\langle \xi,\widetilde{\alpha}_g(\xi-\eta)\rangle)|.
\end{equation}
The Cauchy-Schwarz inequality for the state $\phi$ then implies
$$
|\phi(\langle \eta-\xi,\widetilde{\alpha}_g\eta\rangle)|\leq \phi(\langle \eta-\xi,\widetilde{\alpha}_g\eta\rangle^*\langle \eta-\xi,\widetilde{\alpha}_g\eta\rangle)^{1/2}.
$$
Positivity of $\phi$ and Cauchy-Schwarz for the inner product on $\ell^2(G,\pi(A)')$  gives 
\begin{align*}
\phi(\langle \eta-\xi,\widetilde{\alpha}_g\eta\rangle^*\langle \eta-\xi,\widetilde{\alpha}_g\eta\rangle) & \leq \|\widetilde{\alpha}_g\eta\|_{\ell^2(G,A)} \phi(\langle  \eta-\xi,  \eta-\xi\rangle)^{1/2},
\end{align*}
and hence we get from line \eqref{str norm small} that
$$
|\phi(\langle \eta-\xi,\widetilde{\alpha}_g\eta\rangle)|\leq \|\widetilde{\alpha}_g\eta\|_{\ell^2(G,A)}\phi( \langle  \eta-\xi,  \eta-\xi\rangle)^{1/2}< \frac{\epsilon}{3}.
$$
Similarly, 
$$
|\phi(\langle \xi,\widetilde{\alpha}_g(\xi-\eta)\rangle)|< \frac{\epsilon}{3}.
$$
Hence from line \eqref{start} we get 
$$
|\phi(\langle \eta,\widetilde{\alpha}_g\eta\rangle-\langle \xi,\widetilde{\alpha}_g\xi\rangle)|<\frac{2\epsilon}{3}
$$
and from this and line \eqref{assume} we get 
$$
|\phi\big(\langle \eta,\widetilde{\alpha}_{g}\eta\rangle -1\big)|<\epsilon
$$
for all $\phi\in \Phi$ and $g\in \mathcal{G}$.

Now, as the finite subset $\Phi$ of the state space of $A^{{\op}}$ was arbitrary, and as the states span $(A^{{\op}})^*$, this implies that in the $C^*$-algebra 
$$
B:=\bigoplus_{g\in \mathcal{G}}A^{{\op}}
$$
we have that zero is in the weak closure of the set
$$
\Big\{\big(\langle \eta,\widetilde{\alpha}_{g}\eta\rangle-1_{A^{\op}}\big)_{g\in \mathcal{G}}\in B\mid \eta:G\to A^{{\op}} \text{ finitely supported and }\langle \eta, \eta\rangle\leq 1\Big\}.
$$
Hence by the Hahn-Banach theorem, zero is in the norm closure of the convex hull of this set.  

It follows that for some given $\epsilon>0$ we can find a finite collection $\xi_1,...,\xi_n$ of finitely supported functions $G\to A^{\op}$ and $t_1,...,t_n\in [0,1]$ with $\sum t_i=1$ such that $\langle \xi_i,\xi_i\rangle\leq 1$ for each $i$, and such that 
\begin{equation}\label{cc ineq}
\Big\|1_{A^{\op}}-\sum_{i=1}^n t_i\langle  \xi_i,\widetilde{\alpha}_g\xi_i\rangle\Big\|<\frac\epsilon 2
\end{equation}
for all $g\in \mathcal{G}$.  Define now $h:G\to A^{\op}$ by 
$$
h(g):=\sum_{i=1}^n t_i \langle\xi_i, \widetilde{\alpha}_g\xi_i\rangle.
$$
Then clearly $h$ is finitely supported, and one checks directly that it is positive type.  Hence a GNS-type construction as in  \cite[Proposition 2.5]{Anantharaman-Delaroche:1987os} gives $\xi_0\in \ell^2(G,A^{\op})$ with 
$$
h(g)=\langle \xi_0,\widetilde{\alpha}_g\xi_0\rangle
$$
for all $g\in G$.  Note that $\langle \xi_0,\xi_0\rangle=h(e)\leq 1$, whence $\|\xi_0\|_{\ell^2(G,A^{\op})}\leq 1$.  Hence we may find finitely supported $\xi:G\to A^{{\op}}$ with $\|\xi\|_{\ell^2(G,A^{\op})}\leq 1$ and $\|\xi_0-\xi\|_{\ell^2(G,A^{\op})}< \epsilon/4$.  This gives that $\langle \xi,\xi\rangle \leq 1$ and that for any $g\in \mathcal{G}$, 
\begin{align*}
\|\langle \xi_0,\widetilde{\alpha}_g\xi_0\rangle-\langle \xi,\widetilde{\alpha}_g\xi\rangle\|_{A^{{\op}}}\leq 2 \|\xi_0-\xi\|_{\ell^2(G,A^\op)}<\frac\epsilon 2
\end{align*}
by the Cauchy-Schwarz inequality for Hilbert modules (twice); combined with line \eqref{cc ineq} above and the fact that $\sum_{i=1}^n t_i\langle  \xi_i,\widetilde{\alpha}_g\xi_i\rangle=\langle \xi_0,\widetilde{\alpha}_g\xi_0\rangle$ for all $g\in G$, this completes the proof.
\end{proof}

We now fix some notation.  For a faithful representation $\pi:A\to \Bd(H)$, let 
$$
\widetilde{\pi}:A\to\Bd(H\otimes \ell^2(G)),\quad \widetilde{\pi}(a):v\otimes \delta_g \mapsto \pi(\alpha_{g^{-1}}(a))v\otimes \delta_g
$$
be the usual induced-type representation, so $(\widetilde{\pi},1\otimes \lambda)$ is a covariant pair that integrates to a faithful representation 
$$
\widetilde{\pi}\rtimes (1\otimes \lambda):A\rtimes_{\red} G \to \Bd(H\otimes \ell^2(G)).
$$
For a finite subset $F$ of $G$, let $M_F$ denote $\Bd(\ell^2(F))$, i.e.\ the `$F$-by-$F$ matrices', and for each $g,h\in F$, let $e_{g,h}\in M_F$ denote the corresponding matrix unit.

The following result is contained in the proof of \cite[Lemma 4.2.3]{Brown:2008qy}.

\begin{lemma}\label{cp maps}
With notation as above, let $F$ be a finite subset of $G$.  Then there is a well-defined ccp map determined by the formula
$$
\psi_F:A\rtimes_{\red} G \to A\otimes M_F,\quad a\delta_g \mapsto \sum_{h\in F\cap gF} \alpha_{g^{-1}}(a)\otimes e_{h,g^{-1}h}.
$$
Moreover, there is a well-defined cp map determined by the formula
$$
\phi_F:A\otimes M_F\to A\rtimes_{\red} G, \quad a\otimes e_{g,h}\mapsto \widetilde{\pi}(\alpha_g(a))(1\otimes \lambda_{gh^{-1}}). \eqno\qed
$$
\end{lemma}

Most of the computations in the proof of the following result are inspired by \cite[Lemma 4.3.3]{Brown:2008qy}.

\begin{proof}[Proof of Theorem \ref{com amen ex}]
Let $\pi:A\to \Bd(H)$ be the restriction of a standard form of $A^{**}$ to $A$; we identify $A$ with its image under this representation when convenient and write $\widetilde{\pi}$ for the usual induced representation  
$$
\widetilde{\pi}:A\to \Bd(H\otimes \ell^2(G)).
$$
Let $\pi^{{\op}}:A^{{\op}}\to \Bd(H)$ be the corresponding faithful representation of $A^{{\op}}$ on $H$ coming from the properties of a standard form (see Theorem \ref{hsf the}).   Assuming first that $A$ is nuclear, the commuting representations $\pi$ and $\pi^{{\op}}$ gives rise to a faithful representation 
$$
\sigma:A\otimes A^{{\op}}\to \Bd(H)
$$
on the minimal tensor product of $A$ and $A^{{\op}}$.  Write $B=\sigma(A\otimes A^{{\op}})$, and let $\widetilde{\sigma}$ be the corresponding induced representation
$$
\widetilde{\sigma}:B\to \Bd(H\otimes \ell^2(G)).
$$
Let $\epsilon>0$ and finite subsets $\mathcal{G}\subseteq G$ and $\mathcal{A}\subseteq A$ be given.  We claim that there are a finite subset $F$ of $G$ and ccp maps
$$
\xymatrix{ A\rtimes_{\red} G \ar[dr]^-{\psi} & & B\rtimes_{\red} G \\ & A\otimes M_F \ar[ur]^-{\phi}  & }
$$
such that 
$$
\|\phi\psi(a\delta_g)-\widetilde{\sigma}(a\otimes 1)(1\otimes \lambda_g)\|<\epsilon
$$ 
for all $a\in \mathcal{A}$ and $g\in \mathcal{G}$.  As $A\otimes M_F$ is nuclear, this will suffice to complete the proof.

To prove the claim let $d:=\max_{a\in \mathcal{A}}\|a\|$ and let $\xi:G\to A^{\op}$  have the properties as in Lemma \ref{com amen} with respect to the finite set $\mathcal{G}$, and the constant $\epsilon/d$, so in particular
\begin{equation}\label{small norm}
\|1-\langle \xi,\widetilde{\alpha}_g\xi\rangle \|<\frac{\epsilon}{d}
\end{equation}
for all $g\in \mathcal{G}$.  Let $F\subseteq G$ be a finite set such that $F\cap gF$ contains the support of $\xi$ for all $g\in \mathcal{G}$ (for example, $F=\text{supp}(\xi)\cup\bigcup_{g\in \mathcal{G}}g^{-1}\text{supp}(\xi)$ works).  Define first 
$$
\psi:A \rtimes_{\red} G \to A\otimes M_F
$$
to be the map $\psi_F$ from Lemma \ref{cp maps}, which is ccp.  Define $X\in A^{{\op}}\otimes M_F$ by 
$$
X:=\sum_{h\in F}\alpha_{h^{-1}}(\xi(h))\otimes e_{hh}
$$
and define 
$$
\kappa:A\otimes M_F\to B\otimes M_F,\quad a\mapsto X^*aX
$$
(here we have included $A\otimes M_F$ and $A^{\op}\otimes M_F$ in $B\otimes M_F$ in the natural ways to make sense of the product on the right).  Clearly $\kappa$ is cp.  Finally, let $\phi_F:B\otimes M_F\to B\rtimes_{\red} G$ be as in Lemma \ref{cp maps} applied to the $C^*$-algebra $B$, and define 
$$
\phi:A\otimes M_F\to B\rtimes_{\red} G, \quad \phi:=\phi_F\circ \kappa.
$$
To complete the proof in the nuclear case, it will suffice to show that $\phi$ and $\psi$ have the claimed properties.

Indeed, we already have from Lemma \ref{cp maps} that $\psi$ is ccp.  As we also already know from the same lemma that $\phi$ is cp, to see that it is ccp it suffices to show that $\phi(1)\leq 1$.  For this, we compute that 
\begin{align*}
\phi(1) & =\phi_F\Big(X^*1X\Big)=\phi_F\Big( \sum_{h\in F} \alpha_{h^{-1}}(\xi(h)^*\xi(h))\otimes e_{h,h}\Big)  =\sum_{h\in F}\xi(h)^*\xi(h)=\langle \xi,\xi\rangle \\ & \leq 1
\end{align*}
as claimed.  It remains to show that 
$$
\|\phi\psi(a\delta_g)-\widetilde{\sigma}(a\otimes 1)(1\otimes \lambda_g)\|<\epsilon,
$$ 
or equivalently that 
$$
\|\phi\psi(a\delta_g)-\widetilde{\pi}(a)(1\otimes \lambda_g)\|<\epsilon,
$$ 
for all $a\in \mathcal{A}$ and $g\in \mathcal{G}$.  For this, we compute that 
\begin{align*}
& \phi\psi(a\delta_g)=\phi_F\kappa\Big(\sum_{h\in F\cap gF}\alpha_{h^{-1}}(a)\otimes e_{h,g^{-1}h}\Big) 
\\ & =\phi_F\Big(\Big(\sum_{k\in F}\alpha_{k^{-1}}(\xi(k)^*)\otimes e_{k,k}\Big)\Big(\sum_{h\in F\cap gF}\alpha_{h^{-1}}(a)\otimes e_{h,g^{-1}h}\Big) \Big(\sum_{l\in F}\alpha_{l^{-1}}(\xi(l))\otimes e_{l,l}\Big)\Big)\\
& =\phi_F\Big(\sum_{h\in F\cap gF}\alpha_{h^{-1}}(\xi(h)^*a\alpha_g(\xi(g^{-1}h))\otimes e_{h,g^{-1}h}\Big).
\end{align*}
Hence using the formula for $\phi_F$ in Lemma \ref{cp maps} we get
\begin{align*}
\phi\psi(a\delta_g) & = \sum_{h\in F\cap gF}\widetilde{\sigma}(\xi(h)^*a\alpha_g(\xi(g^{-1}h)))(1\otimes \lambda_g).
\end{align*}
As $F\cap gF$ contains the support of $\xi$ for all $g\in \mathcal{G}$, and as $a$ commutes with the image of $\xi$, we have that this equals 
$$
 \widetilde{\pi^{{\op}}}(\langle \xi,\widetilde{\alpha}_g\xi\rangle) \widetilde{\pi}(a)(1\otimes \lambda_g).
$$
Hence 
$$
\|\phi\psi(a\delta_g)-\widetilde{\pi}(a)(1\otimes \lambda_g)\|\leq (\max_{a\in \mathcal{A}}\|a\|)\|1-\langle \xi,\widetilde{\alpha}_g\xi\rangle\|,
$$
and so we are done in the nuclear case by the inequality in line \eqref{small norm}.

If we only assume $A$ is exact, we can run much of the above proof, replacing $B$ with the $C^*$-subalgebra of $\Bd(H)$ generated by $\pi(A)$ and $\pi^{{\op}}(A^{{\op}})$ to get that for each finite subset $\mathcal{A}$ of $A\rtimes_{\red} G$ and $\epsilon>0$ there are a finite subset $F$ of $G$ and ccp maps $\psi$ and $\phi$ as in the diagram below
$$
\xymatrix{ A\rtimes_{\red} G \ar[rr] \ar[dr]^-\psi & & B\rtimes_{\red} G \ar[r] & \Bd(H\otimes \ell^2(G)) \\ & A\otimes M_F \ar[ur]^\phi & & },
$$
where the horizontal maps are the obvious inclusions, and where the diagram `almost commutes' in the sense that $\|\phi\psi(a)-a\|<\epsilon$ for all $a\in \mathcal{A}$ (having identified $A\rtimes_{\red} G$ with its image in $\Bd(H\otimes \ell^2(G))$ to make sense of this).  Now, as $A\otimes M_F$ is exact, there exists a Hilbert space $H'$ and a nuclear faithful embedding $A\otimes M_F\to \Bd(H')$.  In the diagram
$$
\xymatrix{ A\rtimes_{\red} G \ar[rr] \ar[dr]^-\psi & & B\rtimes_{\red} G \ar[r] & \Bd(H\otimes \ell^2(G)) \\ & A\otimes M_F \ar[ur]^\phi \ar[rr] & & \Bd(H') \ar@{-->}[u] },
$$
the dashed arrow can be filled in with a ccp map by Arveson's extension theorem so that the right hand quadrilateral honestly commutes.  As the map $A\otimes M_F\to \Bd(H')$ is nuclear, the existence of these diagrams gives that $A\rtimes_{\red} G$ is nuclearly embeddable, so exact.

Finally, for exactness of $G$, we note that the above maps restricted to $C^*_r(G)\subseteq A\rtimes_{\red} G$ show that $C^*_r(G)$ is nuclearly embedded in $\Bd(H\otimes \ell^2(G))$ (whether or not $A$ is exact), as the `downwards' map $\psi$ takes image in $M_F(\C)$ when restricted to $C^*_r(G)$.  Hence $C^*_r(G)$ is exact, and thus $G$ is itself exact as it is discrete.
\end{proof}

\begin{remark}\label{nuc rem}
In \cite[Th\'{e}or\`{e}me 4.5]{Anantharaman-Delaroche:1987os}, Anantharaman-Delaroche proves (among other things) that if $A$ is a nuclear $G$-algebra, then $A\rtimes_{\red} G$ is nuclear if and only if the action of $G$ on $A$ is amenable.   On the other hand, if $A$ is unital, nuclear and commutant amenable, we get nuclearity of the inclusion 
$$
A\rtimes_{\red} G \to (A\otimes A^{\op})\rtimes_{\red} G.
$$
It would be interesting if this could somehow be improved to show nuclearity of $A\rtimes_{\red} G$: indeed, we would then get that commutant amenability and amenability are equivalent for all unital and nuclear $G$-algebras, and moreover that for such $G$-algebras, equality of $A\rtimes_{\max}G$ and $A\rtimes_{\red} G$ is equivalent to the conditions in Theorem \ref{nc ex the} when $G$ is exact.  This is related to Remark \ref{ca tens} above.
\end{remark}

\begin{remark}\label{exel rem}
Exel \cite[Definition 20.4]{Exel:2014rp} has introduced a different notion of amenability for $G$-algebras (and more generally for Fell bundles) under the name of the \emph{approximation property}.  He asked whether the existence of a unital $G$-algebra with his approximation property implies exactness of $G$.  The answer is `yes'.  Indeed, the relationship between versions of the approximation property and amenability were extensively studied in \cite{Abadie:2019kc}.  In particular, the results of \cite[Theorem 6.11 and Corollary 6.16]{Abadie:2019kc} imply that Exel's approximation property implies amenability, so Theorem \ref{com amen ex} gives the solution to this question.
\end{remark}

\section{Characterizing equivariant injectivity and the equivariant WEP}\label{inj and wep sec}

In this section, we study $G$-injectivity and the $G$-WEP in more detail.  In particular, we give complete characterizations of both in terms of amenability and the respective non-equivariant versions.  We also collect together many equivalent conditions in the special case of nuclear $G$-algebras and exact groups.

The following definition is partly inspired by work of Anantharaman-Delaroche \cite[Section 2]{Anantharaman-Delaroche:1982aa} in the setting of von Neumann algebras, and of Kirchberg \cite[Proposition 3.1]{Kirchberg:1993aa} in the non-equivariant setting.

\begin{definition}\label{rel inj}
Let $\iota:A\into B$ be a $G$-embedding of $G$-algebras.  The embedding is \emph{relatively $G$\nb-injective} (respectively, \emph{weakly relatively $G$-injective}) if there exists a ccp $G$-map $P:B\onto A$ splitting $\iota$ (respectively, a ccp $G$-map $P:B\to A^{**}$ such that $P\circ \iota:A\to A^{**}$ is the canonical inclusion).
\end{definition}

\begin{remark}\label{inj def 2}
Comparing Definitions \ref{rel inj} and \ref{inj def}, a $G$-algebra $A$ is $G$-injective (respectively, has the $G$-WEP) if any $G$-embedding $A\into B$ of $G$-algebras is relatively $G$-injective (respectively, relatively weakly $G$-injective).
\end{remark}

Our main goals in this section are the following results.  Note the different role played by exactness in the first two theorems: this is essentially due to the fact that if $G$ admits an action on some $A$ such that $A^{**}$ is $G$-injective, then $G$ must be exact; this is not true for the $G$-WEP, however.

\begin{theorem}\label{g wep vs wep}
For a unital $G$-algebra $A$, the following are equivalent:
\begin{enumerate}[(i)]
\item $A$ is amenable and has the WEP;
\item $A$ has the $G$-WEP and $G$ is exact.
\end{enumerate}
\end{theorem}

\begin{theorem}\label{g inj vs inj}
For a $G$-algebra $A$, the following are equivalent:
\begin{enumerate}[(i)]
\item $A$ is amenable and $A^{**}$ is injective;
\item $A^{**}$ is $G$-injective.
\end{enumerate}
Moreover, if there is a unital $G$-algebra satisfying these conditions, then $G$ is exact.
\end{theorem}

The following theorem, which is essentially `just' a compilation of our results and other results of Claire Anantharaman-Delaroche from \cite{Anantharaman-Delaroche:1979aa} and \cite{Anantharaman-Delaroche:1987os}, summarises some of the known facts about amenable actions of exact groups on nuclear $C^*$-algebras, and the relationships to $G$-injectivity and the $G$-WEP.  

\begin{theorem}\label{big nuclear list}
If $G$ is an exact group and $A$ is a nuclear $G$-algebra, then the following assertions are equivalent:
\begin{enumerate}[(i)]
\item \label{bn ada} $A$ is amenable;
\item \label{bn gwep} $A$ has the $G$-WEP;
\item \label{bn ginj} $A^{**}$ is $G$-injective;
\item \label{bn ** boa} $A^{**}$ is strongly amenable;
\item \label{bn ** ada} $A^{**}$ is amenable;
\item \label{bn nuc} $A\rtimes_\red G$ is nuclear;
\item \label{bn cp inj} $A^{**}\bar\rtimes G$ is injective.
\end{enumerate}
Here $A^{**}\bar\rtimes G$ denotes the von Neumann algebra crossed product of $A^{**}$ with $G$.
\end{theorem}

\begin{remark} 
We cannot add the condition that $A$ is strongly amenable  to the equivalent conditions in Theorem~\ref{big nuclear list} by Suzuki's examples in \cite{Suzuki:2018qo}.   We do not know if we can add $A\rtimes_{\max}G=A\rtimes_{\red}G$ to this list of equivalent conditions: compare Remarks \ref{ca tens} and \ref{nuc rem}.
\end{remark}

The proofs will proceed via a series of ancillary lemmas and propositions.  The first few results compare $G$-injectivity and the $G$-WEP to the non-equivariant versions (sometimes also in the presence of amenability).  

\begin{lemma}\label{pro:G-WEP=>WEP}
Let $G$ be a group, $H\sbe G$ a subgroup and $A$ a $G$-algebra. If $A$ has the $G$-WEP (respectively, is $G$-injective), then $A$ also has the $H$-WEP (respectively, is $H$-injective) with the restricted action.
\end{lemma}

\begin{proof}
Let $\pi:A\to \Bd(K)$ be a faithful representation (ignoring the $G$-action).  Equip $B:=\ell^\infty(G,\Bd(K))$ with the translation $G$-action defined by
$$
(\gamma_gf)(h):=f(g^{-1}h)
$$
and then consider the canonical $G$-embedding 
$$
\widetilde\pi\colon A\into B, \quad \tilde\pi(a)(g):=\pi(\alpha_g^{-1}(a)).
$$ 
If $A$ has the $G$-WEP, then there exists a ccp $G$-map $P\colon B\to A^{**}$ with $P\circ\tilde\pi$ equal to the canonical embedding $A\into A^{**}$.   Note that $B$ is $H$-injective as an $H$-algebra with the restricted $H$-action: this is proved in \cite[Remark 6.3]{Buss:2018nm}.  Hence if $A$ embeds $H$-equivariantly into some $H$-algebra $C$, then by the $H$-injectivity of $B$ the dashed arrow below can be filled in 
$$
\xymatrix{ C \ar@{-->}[dr] & \\ A \ar[u] \ar[r]^-{\widetilde{\pi}} & B }
$$
with a ccp $H$-map.  Composing this with $P:B\to A^{**}$ yields the desired ccp $H$-map $C\to A^{**}$ extending the canonical inclusion $A\into A^{**}$. 

The assertion on injectivity can be proved in essentially the same way.
\end{proof}

\begin{corollary}\label{rem:G-WEP-Injectivity}
If a $G$-algebra has the $G$-WEP (respectively, is $G$-injective) then it has the WEP (respectively, is injective). 
\end{corollary}

\begin{proof}
Take $H$ to be the trivial group in Lemma \ref{pro:G-WEP=>WEP}.
\end{proof}

The following result is closely related to \cite[Proposition~4.1]{Anantharaman-Delaroche:1979aa}; we need something a little different, however, so give a direct proof.  The statement essentially says that a non-equivariant ccp splitting can be bootstrapped up to an equivariant splitting in the presence of amenability.

\begin{lemma}\label{non-equi to equi} 
Let $A$ be an amenable  $G$-algebra.  Let $A\subseteq B$ (respectively, $A^{**}\subseteq B$) be a $G$-embedding. Then if the embedding is relatively weakly injective (respectively, relatively injective), it is also relatively weakly $G$-injective (respectively, relatively $G$-injective).
\end{lemma}

\begin{proof}
We will just look at the case where $A^{**}\subseteq B$ is a relatively injective $G$-embedding; the other case is essentially the same.  Relative injectivity gives us a ccp map $\phi:B\to A^{**}$ such that the restriction of $\phi$ to $A^{**}$ is the identity; our task is to replace $\phi$ with a ccp $G$-map without changing it on $A^{**}$.

Let $\ell^2(G,A^{**})$ be as in Definition \ref{standard mod}, equipped with the associated action $\widetilde{\alpha}$ defined there.  For each $b\in B$, define a multiplication-type operator $m(b)$ on $\ell^2(G,A^{**})$ by the formula
$$
(m(b)\xi)(g):=\alpha_g(\phi(\beta_{g^{-1}}(b)))\xi(g).
$$
Then it is not difficult to see that $m$ defines a ccp map $m:B\to \Bd(\ell^2(G,A^{**}))$ from $B$ to the adjointable operators on $\ell^2(G,A^{**})$.  Let now $(\xi_i)_{i\in I}$ be a net as in the definition of amenability, so each $\xi_i$ is a finitely supported function $\xi_i:G\to Z(A^{**})$ such that $\langle \xi_i,\xi_i\rangle \leq 1$ for all $i$, and so that $\langle \xi_i,\widetilde{\alpha}_{g}(\xi_i)\rangle$ converges ultraweakly to $1$ for all $g\in G$.  For each $i$, define a map
$$
\psi_i:B\to A^{**},\quad b\mapsto \langle \xi_i,m(b)\xi_i\rangle.
$$
One then checks that the net $(\psi_i)$ consists of ccp maps, and so, by \cite[Theorem 1.3.7]{Brown:2008qy} and after passing to a subnet if necessary, 
has an ultraweak limit point, which is also a ccp map $\psi:B\to A^{**}$.  We claim that this limit has the right properties.

First, let us check that if $a$ is an element of $A^{**}$, then $\psi(a)=a$.  Indeed, in this case $m(a)$ is just the operator of left-multiplication by $a$, and so we have 
$$
\psi_i(a)=\langle \xi_i,a\xi_i\rangle=\sum_{g\in G}\xi_i(g)^*a\xi_i(g).
$$
for all $i$.  As $\xi$ takes values in $Z(A^{**})$, this just equals $\langle \xi_i,\xi_i\rangle a$, however, which converges ultraweakly to $a$ as $i$ tends to infinity.  

It remains to check that $\psi$ is equivariant.  Let then $b\in B$ and $h\in G$.  Then 
\begin{align*}
\psi_i(\beta_h(b))=\langle \xi_i,m(\beta_h(b))\xi_i\rangle=\sum_{g\in G} \xi_i(g)^*\alpha_g\phi(\beta_{g^{-1}h}(b))\xi_i(g).
\end{align*}
Making the substitution $k=h^{-1}g$, this equals 
\begin{align*}
\sum_{k\in G} \xi_i(hk)^*\alpha_{hk}\phi(\beta_{k^{-1}}(b))\xi_i(hk) & =\alpha_h\Big(\sum_{k\in G} (\widetilde{\alpha}_{h^{-1}}\xi_i)(k)^*\alpha_{k}\phi(\beta_{k^{-1}}(b))(\widetilde{\alpha}_{h^{-1}}\xi_i)(k)\Big) \\ & = \alpha_h(\langle \widetilde{\alpha}_{h^{-1}}\xi_i,m(b)\widetilde{\alpha}_{h^{-1}}\xi_i\rangle).
\end{align*}
To prove equivariance, it thus suffices to show that 
\begin{equation}\label{equi diff}
\langle \widetilde{\alpha}_{h^{-1}}\xi_i,m(b)\widetilde{\alpha}_{h^{-1}}\xi_i\rangle - \langle \xi_i,m(b)\xi_i\rangle
\end{equation}
tends ultraweakly to zero.  This follows as we have the identity 
$$
\langle \widetilde{\alpha}_{h^{-1}}\xi_i - \xi_i,\widetilde{\alpha}_{h^{-1}}\xi_i - \xi_i\rangle=\alpha_h(\langle \xi_i,\xi_i\rangle) +\langle \xi_i,\xi_i\rangle-\langle \xi_i,\widetilde{\alpha}_{h^{-1}}\xi_i \rangle-\langle \widetilde{\alpha}_{h^{-1}}\xi_i ,\xi_i\rangle,
$$
and the right hand side tends ultraweakly to zero.  The expression in line \eqref{equi diff} therefore tends ultrweakly to zero using appropriate versions of the Cauchy-Schwarz inequality similarly to the proof of Lemma \ref{com amen}.
\end{proof}

\begin{corollary}\label{wep to g wep}
Let $A$ be an amenable  $G$-algebra.  Then if $A$ has the WEP (respectively, if $A^{**}$ is injective), then $A$ has the $G$-WEP (respectively, $A^{**}$ is $G$-injective). \qed
\end{corollary}

The next lemma is closely related to Lemma \ref{unital lem}.

\begin{lemma}\label{lem:G-WEP=>UCP}
Let $A$ be a $G$-algebra and assume $A$ is $G$-injective (respectively, has the $G$-WEP).  Let $C$ be any unital $G$-algebra.  Then there is a ucp $G$-map $C\to Z(A)$ (respectively, $C\to Z(A^{**})$).
\end{lemma}

\begin{proof}
Assume first that $A$ has the $G$-WEP. Let $B$ be the $G$-algebra $B:=C\otimes A$ equipped with the diagonal action $\gamma\otimes \alpha$ where $\gamma$ denotes the action on $C$, and $\alpha$ the action on $A$.  Consider the canonical $G$-embedding 
$$
\iota\colon A\into B, \quad a\mapsto 1\otimes a.
$$
Since $A$ has the $G$-WEP, there is a ccp $G$-map $P\colon B\to A^{**}$ such that $P\circ\iota$ coincides with the canonical embedding $A\into A^{**}$.  Fix an approximate unit $(e_i)_{i\in I}$ for $A$, and for each $i$, define 
$$
P_i:C\to A^{**},\quad c\mapsto P(c\otimes e_i).
$$
The net $(P_i)$ of ccp maps has a point-ultraweak limit, say $Q:C\to A^{**}$, which we claim is the required map.  As $Q$ is automatically ccp, we must check three things: that $Q$ has image in $Z(A^{**})$; $Q$ is unital; and that $Q$ is equivariant.

Indeed, note first that the subalgebra $\{1\otimes a\mid a\in A\}$ of $B$ is in the multiplicative domain of $P$, whence for each $a\in A$, $c\in C$, and $i\in I$,
$$
aP_i(c)-P_i(c)a=aP(c\otimes e_i)-P(c\otimes e_i)a=P(c\otimes (ae_i-e_ia)),
$$
which tends to zero (in norm) as $i$ tends to infinity.  Hence the image of $Q$ commutes with $A$, and thus with all of $A^{**}$, so is central.  To see that $Q$ is unital, note that $P_i(1)=P(1\otimes e_i)=e_i$, and that any approximate unit for $A$ converges ultraweakly to the unit of $A^{**}$.  Finally, to see that $Q$ is equivariant, we note that $P$ is equivariant, whence if $\alpha$ denotes the $G$-actions on both $A$ and $A^{**}$, and $\gamma$ the action on $C$, then for any $c\in C$ and $i\in I$, we have 
$$
\alpha_g(P_i(c))-P_i(\gamma_g(c))=\alpha_g(P(c\otimes e_i))-P(\gamma_g(c)\otimes e_i))=P(\gamma_g(c)\otimes (\alpha_g(e_i)-e_i)),
$$
which tends to zero (in norm) as $i$ tends to infinity.

The case where $A$ is $G$-injective is similar and easier because in this case $A$ is now unital (see Lemma \ref{uni rem} above).  Indeed, consider again the same embedding $\iota\colon A\into B$ as above. Notice that $A$, $B$ and $\iota$ are unital. Since $A$ is now $G$-injective, we get a ucp $G$-map $P\colon B\to A$ satisfying $P\circ\iota=\id_A$. Since $\iota$ is unital, so is $P$ and the same argument as before shows that $P(C)\sbe Z(A)$.  Composing with the canonical embedding $C\into B$ yields the desired ucp $G$-map $Q:C\to Z(A)$.
\end{proof}

\begin{corollary}\label{cor:G-WEP=>amenable}
If $G$ is an exact group and $A$ is a $G$-algebra that is $G$-injective (respectively, has the $G$-WEP), then it is strongly amenable (respectively, amenable).
\end{corollary}

\begin{proof}
This follows directly from Proposition \ref{G-exact-char-amenable} and Lemma \ref{lem:G-WEP=>UCP} in the special case $C=\ell^\infty(G)$.
\end{proof}

We are finally ready to prove our main theorems from the start of this section.

\begin{proof}[Proof of Theorem \ref{g wep vs wep}]
Assume first that $A$ has the $G$-WEP and that $G$ is exact.  Then $A$ has the WEP by Corollary \ref{rem:G-WEP-Injectivity}, and is amenable by Corollary \ref{cor:G-WEP=>amenable}.

Conversely, if $A$ has the WEP and is amenable, then it has the $G$-WEP by Corollary \ref{wep to g wep}.  As $A$ is unital, the existence of an amenable  action implies that $G$ is exact by Theorem \ref{com amen ex}.
\end{proof}

\begin{remark}\label{g wep to wep rem}
Say $G$ is an exact group.  Then the above proof shows that for any (not necessarily unital) $G$-algebra $A$, the following are equivalent:
\begin{enumerate}[(i)]
\item $A$ is amenable and has the WEP;
\item $A$ has the $G$-WEP.
\end{enumerate}
In other words, if we are willing to assume exactness, we can drop the unitality assumption from Theorem \ref{g wep vs wep}.  The above equivalences do not hold (for unital algebras) in the non-exact case: indeed, $A=\ell^\infty(G)$ is $G$-injective, so in particular has the $G$-WEP, but it is not amenable if $G$ is not exact.  On the other hand, the equivalence of
\begin{enumerate}[(i)]
\item $A$ is amenable and has the WEP, and
\item $A$ has the $G$-WEP and $G$ is exact
\end{enumerate}
from Theorem \ref{g wep vs wep} do not hold in the non-unital case: $A=C_0(G)$ is amenable and has the WEP, so satisfies the first condition whether $G$ is exact or not.
\end{remark}

\begin{proof}[Proof of Theorem \ref{g inj vs inj}]
Assume first that $A^{**}$ is $G$-injective.  Then $A^{**}$ is injective by Corollary \ref{rem:G-WEP-Injectivity}.  Moreover, $G$-injectivity of $A^{**}$ gives an equivariant conditional expectation $\ell^\infty(G,A^{**})\to A^{**}$ as in \cite[Th\'{e}or\`{e}me 3.3, part (e)]{Anantharaman-Delaroche:1987os}, which implies amenability.  

Conversely, say $A^{**}$ is injective and amenable.  Then $A^{**}$ is $G$-injective by Corollary \ref{wep to g wep}.

Finally, note, if $A$ is unital, then amenability of $A$ implies exactness of $G$ by Theorem \ref{com amen ex}.
\end{proof}

\begin{proof}[Proof of Theorem \ref{big nuclear list}]
As $A$ is nuclear, it has the WEP (see for example \cite[Corollary 3.6.8]{Brown:2008qy}).  Hence the equivalence of \eqref{bn ada} and \eqref{bn gwep} follows from Theorem \ref{g wep vs wep}.  Similarly, if $A$ is nuclear than $A^{**}$ is injective, and so the equivalence of \eqref{bn ginj} and \eqref{bn ada} follows from Theorem \ref{g inj vs inj}.

The fact that \eqref{bn ginj} implies \eqref{bn ** boa} follows from Corollary~\ref{cor:G-WEP=>amenable}, and \eqref{bn ** boa} implies \eqref{bn ** ada} is trivial.  Assuming that $A^{**}$ is amenable, note that the universal property of $A^{****}$ gives a normal equivariant surjective $*$-homomorphism $A^{****}\onto A^{**}$ splitting the canonical inclusion $A^{**}\into A^{****}$.  This restricts to a normal $*$-homomorphism $Z(A^{****})\onto Z(A^{**})$, from which it follows that $A$ is amenable, giving \eqref{bn ada}.  We now have that conditions \eqref{bn ada} through \eqref{bn ** ada} are equivalent.  

Finally, note that the equivalence of \eqref{bn ada} to both \eqref{bn nuc} and \eqref{bn cp inj} was established by Anantharaman-Delaroche in \cite[Th\'{e}or\`{e}me 4.5]{Anantharaman-Delaroche:1987os}. 
\end{proof}

%

\begin{remark}\label{g wep vs cp wep}
In Theorem \ref{g wep vs wep}, we have compared the $G$-WEP for $A$ to the WEP for $A$ and amenability type conditions.  It is also natural to compare the $G$-WEP for $A$ to the WEP for the crossed products $A\rtimes_{\red} G$ and $A\rtimes_{\max} G$.

We first note that if $A$ is amenable and has the WEP, then $A\rtimes_\max G=A\rtimes_\red G$ has the WEP.   This was proved by Bhattarcharya and Farenick in \cite{Bhattacharya:2013sj}.  One can also give a short argument using that a $C^*$-algebra $B$ has the WEP if and only if 
$$
B\otimes_\max C^*(\free_\infty)=B\otimes C^*(\free_\infty)
$$
(see for example \cite[Corollary 13.2.5]{Brown:2008qy}).  If $G$ is exact, we already know from Theorem \ref{g wep vs wep} that $A$ has the $G$-WEP if and only if $A$ is amenable and has the WEP.  In particular, if $A$ has the $G$-WEP and $G$ is exact, then $A\rtimes_\max G=A\rtimes_\red G$ has the WEP.

On the other hand, is is shown in \cite[Proposition~5.4]{Buss:2018nm} that if $A\rtimes_\inj G$ has the WEP, then $A\rtimes_\max G=A\rtimes_\inj G$. Moreover, if $A\rtimes_\inj G$ has the WEP, then so does $A$ because $A$ is a $C^*$-subalgebra of $A\rtimes_\inj G$ with a conditional expectation $A\rtimes_\inj G\onto A$.  Hence if $G$ is exact and $A\rtimes_\red G$ has the WEP, then we have that $A\rtimes_\max G=A\rtimes_\inj G=A\rtimes_\red G$ and that $A$ has the WEP. 

Summarizing the above discussion, if $G$ is an exact group, then we know that 
$$
A \text{ has $G$-WEP} \quad \Rightarrow \quad  A\rtimes_{\red}G \text{ has WEP}
$$
and that 
$$
A\rtimes_{\red}G \text{ has WEP} \quad \Rightarrow \quad A \text{ has WEP and } A\rtimes_{\max}G=A\rtimes_r G.
$$
If $A$ is commutative, the latter condition also implies that $A$ is amenable  by Theorem \ref{com ex the}, and therefore that $A$ has $G$-WEP by Theorem \ref{g wep vs wep}.  The precise situation is not clear in general, however.

If $G$ is not exact, then things are murky.  For example, it is not clear whether the \cstar{}algebras $\ell^\infty(G)\rtimes_\max G$ or $\ell^\infty(G)\rtimes_\red G$ could have the WEP if the $G$-action on $A=\ell^\infty(G)$ is not amenable.
\end{remark}

\section{Hamana's theory of injective envelopes}\label{hamana sec}

In this section, we discuss the relation of the notion of injectivity that we have been using with Hamana's from \cite{Hamana:1985aa} (they turn out to be the same, fortunately).  We also use some of our work above to address some questions about injective envelopes that seem to be of interest in their own right.

Hamana's definition of $G$-injectivity is as follows.  Consider a diagram 
\begin{equation}\label{more inj}
\xymatrix{ C \ar@{-->}[dr]^-{\widetilde{\phi}} & \\ A \ar@{_(->}[u]^-\iota \ar[r]^-{\phi} & B }
\end{equation}
where $B$, $C$, and $A$ are operator systems equipped with $G$-actions by complete order automorphisms, $\iota$ is a complete order injection, and $\phi$ is a ucp $G$-map.  Then $B$ is \emph{$G$-injective} if the dashed arrow can be filled in with a ucp $G$-map.  

On the other hand, in Definition \ref{inj def}, we say that a $G$-$C^*$-algebra is \emph{injective} if in a diagram of the form \eqref{more inj} where $C$ is a $G$-$C^*$-algebra, $\iota$ is an injective equivariant $*$-homomorphism, and $\phi$ is the identity map, the dashed arrow can be filled in with a ccp $G$-map.

Now, both Hamana's definition and our definition make sense for unital $G$-algebras.  Fortunately, the two notions coincide (even with respect to their domains of definition: this follows as injectivity of a $G$-algebra $B$ in our sense forces $B$ to be unital by Lemma \ref{uni rem}, and injectivity of a $G$-operator system $B$ in Hamana's sense forces $B$ to admit a structure of a (unital) $G$-algebra by the proof of \cite[Theorem 15.2]{Paulsen:2003ib}).

\begin{proposition}\label{hamana vs inj}
A unital $G$-$C^*$-algebra $B$ is injective in the sense of Definition \ref{inj def} if and only if it is injective in the sense of Hamana.
\end{proposition}

\begin{proof}
First assume that $B$ satisfies Definition \ref{inj def}, so is in particular unital by Remark \ref{uni rem}.  In \cite[Corollary 2.4]{Buss:2018nm} (compare also \cite[Lemma 2.2]{Hamana:1985aa}) it is shown that any (unital) $G$-algebra $B$ admits a (unital) embedding $B\to B_H$ into a $G$-algebra $B_H$ that is injective in Hamana's sense.  Definition \ref{inj def} gives an equivariant ucp map $E:B_H\to B$ splitting this inclusion.  Consider now a diagram as in line \eqref{more inj} where $\phi$ and $\iota$ satisfy the conditions in Hamana's definition of injectivity.  Consider the diagram
$$
\xymatrix{ C \ar@{-->}[dr] \ar@{-->}[drr]^-{\widetilde{\psi}} & &  \\ A \ar@{_(->}[u]^-\iota \ar[r]_-{\phi} & B \ar[r] & B_H }.
$$
As $B_H$ is injective in Hamana's sense, the long diagonal arrow can be filled in with an equivariant ucp map, say $\widetilde{\psi}$.  The required map $\widetilde{\phi}$ can then be defined by $\widetilde{\phi}:=E\circ \widetilde{\psi}$; it is not difficult to check that this works.

Conversely, say $B$ is injective in Hamana's sense.  We need to show that any injective equivariant $*$-homomorphism $B\to C$ admits an equivariant ccp splitting.  We have an extended diagram
$$
\xymatrix{ \widetilde{C} \ar@{-->}[dr] & \\ \widetilde{B} \ar@{_(->}[u] \ar[r]^-{\phi} & B }
$$
where the vertical map is the unitisation of the map we started with, and  the horizontal map is the canonical projection of the unitisation of a unital $C^*$-algebra onto the original algebra (which is a $*$-homomorphism).  Thanks to Hamana's definition, the dashed arrow can be filled in with a ucp $G$-map; the restriction of this arrow to $C$ is the required map.
\end{proof}

We now turn to $G$-injective envelopes.  Recall that in \cite[Theorem 2.5]{Hamana:1985aa}, Hamana proves that every $G$-operator system (and in particular, every unital $G$-algebra) $A$ has a \emph{$G$-injective envelope} $I_G(A)$.  This is a $G$-algebra $I_G(A)$ which is $G$-injective, equipped with a canonical unital $G$-embedding $A\to I_G(A)$, and has the universal property that whenever $A\to B$ is a ucp $G$-map into an injective operator system, there is a unique equivariant ucp extension $I_G(A)\to B$.

The following theorem provides a nice addition to the equivalent conditions in Theorem \ref{big nuclear list}.

\begin{theorem}\label{inj env}
Let $G$ be an exact group, and let $A$ be a nuclear $G$-algebra.  The following are equivalent:
\begin{enumerate}[(i)]
\item \label{env inj} $A^{**}$ is $G$-injective;
\item \label{env wep} $A$ has the $G$-WEP;
\item \label{env emb} there is a $G$-embedding $I_G(A)\into A^{**}$ extending the inclusion $A\into A^{**}$;
\item \label{env rel inj} the inclusion $A\into I_G(A)$ is relatively weakly $G$-injective in the sense of Definition \ref{rel inj}.
\end{enumerate}
\end{theorem}

\begin{proof}
The equivalence of \eqref{env inj} and \eqref{env wep} is already proved in Theorem \ref{big nuclear list}.  Starting with \eqref{env inj}, note that if $A^{**}$ is $G$-injective, then the universal property of $I_G(A)$
implies that we have a $G$-embedding $I_G(A)\into A^{**}$ extending the canonical embedding $A\into A^{**}$. Hence \eqref{env inj} implies \eqref{env emb}. It is clear that \eqref{env emb} implies \eqref{env rel inj}.  

Finally, we claim that \eqref{env rel inj} implies \eqref{env wep}. Indeed, if $A$ embeds into some $G$-algebra $B$, since $I_G(A)$ is $G$-injective, the inclusion $A\into I_G(A)$ extends to a ccp $G$-map $B\to I_G(A)$. Composing this with the map $I_G(A)\to A^{**}$ given by \eqref{env rel inj} yields the desired ccp $G$-map $B\to A^{**}$ extending the inclusion $A\into A^{**}$.
\end{proof}

We conclude this section with two results that can be seen as generalizations of results of Kalantar and Kennedy in their seminal work on the Furstenberg boundary \cite{Kalantar:2014sp}.  The only new idea needed for the proofs in both cases is Corollary \ref{cor:G-WEP=>amenable}.

The first generalization of the work of Kalantar and Kennedy is as follows.  In \cite[Theorem 4.5]{Kalantar:2014sp}, Kalantar and Kennedy prove that $G$ is exact if only if its action on the Furstenberg boundary $\partial_F G$ is amenable. Recalling that $\cont(\partial_F G)$ is the $G$-injective envelope of $\C$ (see \cite{Kalantar:2014sp}*{Theorem~3.11}), the following result is a natural extension.

\begin{theorem}\label{inj env ex}
The following are equivalent for a discrete group $G$.
\begin{enumerate}[(i)]
\item $G$ is exact;
\item $G$ acts strongly amenably (respectively, amenably, or C-amenably) on the injective envelope $I_G(A)$ of every nonzero $G$-algebra $A$.
\item $G$ acts strongly amenably (respectively, amenably, or C-amenably) on the injective envelope $I_G(A)$ of some nonzero $G$-algebra $A$.
\end{enumerate}
\end{theorem}

\begin{proof}
Assume $G$ is exact, and let $A$ be a $G$-algebra.  Then as $I_G(A)$ is $G$-injective, Corollary \ref{cor:G-WEP=>amenable} gives that $I_G(A)$ is strongly amenable. Hence (i) implies (ii).  Since every injective $C^*$-algebra is unital by  Lemma \ref{uni rem}, 
 (iii) implies (i) follows from Theorem \ref{com amen ex}.\end{proof}

The second concerns a conjecture of Ozawa.  In \cite{Ozawa:2007ty}, Ozawa conjectures that every exact \cstar{}algebra $B$ embeds into a nuclear \cstar{}algebra $N(B)$ with $B\sbe N(B)\sbe I(B)$.  Here $I(B)$ denotes the \emph{injective envelope} of $B$, see \cite{Hamana:1979aa}, which is the natural non-equivariant version of the $G$-injective envelope discussed above.

The above conjecture was established for $B=C^*_r(G)$ for any discrete group $G$ by Kalantar and Kennedy in \cite[Theorem 1.3]{Kalantar:2014sp} using the Furstenberg boundary $\partial_F G$.

Using the above observations we can prove Ozawa's conjecture for all crossed products of commutative $G$-algebras by exact discrete groups.

\begin{corollary}\label{oz con cor}
Ozawa's conjecture holds for all \cstar{}algebras $B$ of the form $B=A\rtimes_\red G$, where $A$ is a commutative $G$-algebra and $G$ is an exact group.
\end{corollary}

\begin{proof}
Let $I_G(A)$ be the $G$-injective envelope of $A$. Since $A$ is commutative, so is $I_G(A)$.  In particular, $I_G(A)$ is a nuclear \cstar{}algebra.  By Theorem \ref{inj env ex} (or Corollary~\ref{cor:G-WEP=>amenable}), $I_G(A)$ is (strongly) amenable, so that the crossed product $I_G(A)\rtimes_\red G$ is nuclear by \cite[Th\'{e}or\`{e}me 4.5]{Anantharaman-Delaroche:1987os} (see also Theorem \ref{big nuclear list} above).  On the other hand, by Hamana's results in \cite[Theorem 3.4]{Hamana:1985aa} we have that 
$$
A\rtimes_\red G\sbe I_G(A)\rtimes_\red G\sbe I(A\rtimes_\red G),
$$
so that Ozawa's conjecture holds with $N(B)=I_G(A)\rtimes_\red G$.
\end{proof}

The above proof carries over to every $G$-algebra $A$ for which $I_G(A)$ is nuclear.  However, injective $C^*$-algebras are rarely nuclear outside of the commutative case, so we thought it seemed simpler to state the result when $A$ (and therefore also $I_G(A)$) is commutative.


\end{document}